\AddToHook{package/hyperref/before}{\RequirePackage{float}}
\documentclass[pdflatex,sn-mathphys-num]{sn-jnl}
\usepackage[utf8]{inputenc}
\usepackage[T1]{fontenc}
\usepackage[american]{babel}
\usepackage{amsmath,amssymb,amsfonts,amsthm}
\usepackage{enumitem}
\usepackage{booktabs}
\usepackage{graphicx}
\usepackage{placeins}
\usepackage{float}

\floatstyle{ruled}
\newfloat{algorithm}{htbp}{loa}
\floatname{algorithm}{Algorithm}
\usepackage{tikz}
\usetikzlibrary{arrows.meta}
\hypersetup{hidelinks,bookmarksdepth=2}
\IfFileExists{stmaryrd.sty}{\usepackage{stmaryrd}}{%
  \newcommand{\llbracket}{\mathopen{[\![}}%
  \newcommand{\rrbracket}{\mathclose{]\!]}}}

\theoremstyle{thmstyleone}
\newtheorem{theorem}{Theorem}[section]
\newtheorem{prop}[theorem]{Proposition}

\newtheorem{corollary}[theorem]{Corollary}
\theoremstyle{thmstylethree}

\theoremstyle{thmstyletwo}
\newtheorem{remark}[theorem]{Remark}
\newtheorem*{remark*}{Remark}

\DeclareMathOperator{\Tr}{Tr}

\DeclareMathOperator{\rank}{rank}

\DeclareMathOperator*{\argmin}{arg\,min}

\newcommand{\R}{\mathbb{R}}
\newcommand{\one}{\mathbf{1}}
\newcommand{\X}{\mathcal{X}}
\newcommand{\Dphi}{\mathcal{D}_{\phi}^{\pi}}
\newcommand{\norm}[1]{\left\lVert #1\right\rVert}
\newcommand{\inner}[2]{\left\langle #1,#2\right\rangle}
\newcommand{\Fpi}[1]{\norm{#1}_{F,\pi}}
\newcommand{\opnorm}[1]{\norm{#1}_{2\to2,\pi}}
\numberwithin{equation}{section}
\allowdisplaybreaks
\begin{document}

\title[Bregman projections and partition design]{Bregman projections and partition design for accelerating finite Markov chains}

\author*[1]{\fnm{Michael C.H.} \sur{Choi}}\email{mchchoi@nus.edu.sg}
\author[1]{\fnm{Youjia} \sur{Wang}}\email{e1124868@u.nus.edu}

\affil[1]{\orgdiv{Department of Statistics and Data Science},
\orgname{National University of Singapore},
\orgaddress{\street{Level 7, 6 Science Drive 2},
\postcode{117546}, \country{Singapore}}}

\abstract{We study Gibbs averaging of finite Markov chains with a prescribed stationary distribution $\pi$. Given a transition matrix $P$ and the Gibbs kernel $G$ induced by a state-space partition, we prove that $GPG$ is an information projection for every admissible Bregman generator $\phi$, with divergences defined using transition probabilities relative to $\pi$. A Pythagorean identity characterizes this projection and recovers the relative entropy and weighted Frobenius cases. A dual characterization identifies $GPG$ as a kernel closest to stationarity among those with the same projection chain as $P$; for relative entropy, it maximizes the entropy rate. The identities also give monotonicity under partition coarsening. For every reversible ergodic baseline $P$, Gibbs averaging cannot increase the squared weighted Frobenius distance to stationarity at any positive integer time, and the inequality is strict unless $P=GPG$. The corresponding mixing time therefore cannot increase. To choose the partition, we identify the divergence of $G$ from stationarity as the exact worst-case divergence of $GPG$ over all $\pi$-stationary baselines $P$. This objective depends only on the block masses. We give an exact algorithm for concave block costs, including those
arising from relative entropy and power generators $\phi(u)=u^p$
with $1<p<2$, and a greedy approximation algorithm with a provable
guarantee for $p>2$. For power generators $\phi(u)=u^p$ with $p>1$, the objective
is proportional to Tsallis entropy of order $2-p$:
its minimization favors imbalanced block masses when $1<p<2$
and balanced block masses when $p>2$, while all $k$-block
partitions have the same objective value when $p=2$.
A numerical experiment compares the resulting projected chains
with a local Metropolis sampler on a multi-modal $\pi$.}



\maketitle

\section{Introduction}\label{sec:introduction}

Consider a transition matrix $P$ on a finite state space $\X$ with a prescribed stationary distribution $\pi$ of full support. Given a partition $\mathcal O$ of $\X$, let $G$ be the Gibbs kernel that draws from $\pi$ conditional on the block containing the current state. The composition $GPG$ preserves $\pi$ and adds conditional resampling before and after a transition of $P$. Every finite partition can be realized as the orbit partition of a finite group, so this construction also has a group-averaging interpretation. We write $\Pi$ for the transition matrix with every row equal to $\pi$; the precise definitions are given in Section~\ref{sec:preliminaries}.

This construction raises two questions. For a fixed partition, how does averaging change the distance of a transition matrix from stationarity, and how does this change affect convergence under iteration? When the partition is allowed to vary, which choice gives the best improvement according to a specified information-theoretic distance? These questions are relevant to Markov chain Monte Carlo, where conditional updates can supplement a baseline sampler. They also connect approximation of transition matrices with a combinatorial design problem: selecting the blocks on which averaging takes place.

Information geometry provides a framework for studying structural constraints on Markov transition matrices. \citet{WolferWatanabe2021} investigate the exponential and mixture geometry of reversible Markov chains and derive information projections onto the reversible family. \citet{WolferWatanabe2024} develop a geometric account of lumping and Markov embeddings, including their relation to information projections and data processing. \citet{WatanabeWolfer2026} characterize when families of lumpable stochastic matrices form exponential families. These works explain how properties of a chain are reflected in the geometry of families of transition matrices.

Bregman divergences provide a common setting for relative entropy and quadratic approximation. Starting from the convex-analytic construction of \citet{Bregman1967}, \citet{DhillonTropp2007} develop matrix nearness problems and projection algorithms under Bregman divergences. \citet{Adamcik2014} studies their geometry and their use in pooling probability distributions, while \citet{NockEtAl2013} develop matrix Bregman divergences for portfolio selection. In the present paper, we define Bregman divergences between transition matrices of $\pi$-stationary Markov chains by applying a scalar divergence to transition densities relative to $\pi$, with reference weights $\pi\otimes\pi$. Specifically, for a convex continuous generator $\phi:[0,\infty)\to\R$ that is continuously differentiable on $(0,\infty)$, the scalar Bregman divergence is
\[
B_\phi(u,v):=\phi(u)-\phi(v)-\phi'(v)(u-v),
\qquad u\ge0,\quad v>0.
\]
For $\pi$-stationary transition matrices $P$ and $Q$, we define
\[
\Dphi(P\|Q)
:=\sum_{x,y\in\X}\pi(x)\pi(y)
B_\phi\!\left(\frac{P(x,y)}{\pi(y)},\frac{Q(x,y)}{\pi(y)}\right),
\]
with the boundary conventions of Section~\ref{sec:preliminaries} when $Q(x,y)=0$. A special case of Bregman divergence is the squared-Frobenius norm between transition matrices. For a real matrix $A$ indexed by $\X\times\X$, its squared weighted Frobenius norm is
\[
\Fpi{A}^2:=\sum_{x,y\in\X}\frac{\pi(x)}{\pi(y)}A(x,y)^2.
\]
The quadratic generator $\phi(u):=(u-1)^2$ gives $\Dphi(P\|Q)=\Fpi{P-Q}^2$. Thus $\Dphi(P\|\Pi)$ measures divergence from stationarity, and its quadratic specialization is $\Fpi{P-\Pi}^2$.

The projection identities are closely related to established conditional-expectation results. \citet{BanerjeeGuoWang2005} show that conditional expectation is an optimal predictor under Bregman loss. For block-average matrix approximation, \citet[Section~3.1, Lemma~3 and Theorem~2]{BanerjeeEtAl2007} establish a Pythagorean identity and the corresponding two projection characterizations. For our choice of transition densities and reference weights, the conditional block mean is the density of $GPG$. We specialize this geometry to $\pi$-stationary transition matrices and use it to study convergence and partition design.

Our work also builds on the study of group-averaged Markov chains in \citet{ChoiWang2025,ChoiLimWang2025}. In particular, \citet{ChoiLimWang2025} establish the KL information-projection characterization of $GPG$, its reduction to the projection chain, and the KL-optimal Gibbs partition for a prescribed number of blocks. Subsequent work considers partition design for several averaging constructions. \citet{LimChoiTwoBlock2026} study two-block partitions under KL and Frobenius objectives and develop algorithms using difference-of-submodular decompositions. \citet{LimChoiAdditive2026} analyze additive mixtures of a baseline kernel and a Gibbs kernel. \citet{ChoiWangSpectral2026} develop spectral methods for selecting $k$-block partitions, using weighted $k$-means to round spectral embeddings. The results below extend the KL projection and partition principles to admissible Bregman generators and give a comparison of convergence at every positive integer time in the quadratic case.

Information-theoretic partition selection also appears in \citet{LaiChoi2026}, who study coordinate subsets and partitions of multivariate Markov chains through entropy and KL objectives, using submodular and supermodular structure to obtain greedy algorithms with guarantees. Their optimization variables are subsets and partitions of the coordinate set. Here we select blocks of the state space to define Gibbs averaging kernels, and use Bregman distance to stationarity to guide that selection.

Our main contributions are as follows. Throughout, we use Bregman divergences as distance measures for transition matrices and as objective functions for designing optimal or approximately optimal partitions.
\begin{itemize}[leftmargin=*]
\item \textbf{Bregman projections for stationary transition matrices.}
We define a $\pi$-weighted Bregman divergence that includes relative entropy and the squared-Frobenius norm. For every admissible generator, we prove that $GPG$ minimizes the divergence from $P$ over the stationary kernels whose transition densities are constant on block pairs. A Pythagorean identity gives the exact approximation error, and the minimizer is unique when the generator is strictly convex. Reversibility of $P$ is not required.

\item \textbf{Dual projection, reduction, and partition coarsening.}
We show that $GPG$ also minimizes the divergence to $\Pi$ among stationary kernels with the same projection chain as $P$. For relative entropy, it maximizes the entropy rate within this class. The divergence of every positive power of $GPG$ reduces exactly to that of the corresponding power of the projection chain. Nested partitions yield a telescoping Pythagorean identity and monotonicity of the divergence under partition coarsening.

\item \textbf{Frobenius improvement at every positive time.}
For every reversible ergodic $P$, every partition into $k$ nonempty blocks, and every positive integer $t$, we prove
\[
\Fpi{(GPG)^t-\Pi}^2\le\Fpi{P^t-\Pi}^2.
\]
The inequality is strict at every such time unless $P=GPG$. Consequently, the mixing time defined by the squared weighted Frobenius distance cannot increase. The comparison holds for every $k\in\llbracket2,N\rrbracket$, without a commutation assumption on $P$ and $G$.

\item \textbf{An exact worst-case objective for partition design.}
We identify $\Dphi(G\|\Pi)$ as the maximum of $\Dphi(GPG\|\Pi)$ over all $\pi$-stationary baselines $P$, attained at the identity kernel. Minimizing this quantity therefore gives a partition that optimizes a divergence bound valid for every baseline. We express the objective as a sum of functions of the block masses, making it possible to choose the partition without using the entries of $P$.

\item \textbf{Exact and approximate partition algorithms.}
For concave block costs, we prove that an optimal $k$-block partition consists of the $k-1$ states of smallest stationary mass as singleton blocks and one block containing the remaining states. This includes relative entropy and power generators $\phi(u)=u^p$ with $1<p<2$. For $p>2$, we give an explicit approximation guarantee for a greedy balancing algorithm. The quadratic case is independent of the partition when $k$ is fixed, separating the regimes that favor imbalanced and balanced block masses. A numerical experiment on $100$ states compares the resulting four-block chains with a local Metropolis sampler.
\end{itemize}

The partition objectives guide the design of accelerated chains through a bound on divergence; their optimizers need not minimize a mixing time. The convergence comparison counts transitions of the respective kernels. A computational gain also depends on the cost of constructing the partition and sampling from the conditional distributions defining $G$.

The remainder of the paper is organized as follows. Section~\ref{sec:preliminaries} introduces the divergences, Gibbs kernels, and group-averaging formulation. Section~\ref{sec:projection} proves the projection identities and their consequences. Section~\ref{sec:all-time} establishes the Frobenius comparison and the mixing-time bound. Section~\ref{sec:partition-tuning} develops the partition objectives and algorithms. Section~\ref{sec:numerical} presents the numerical experiment.

\section{Preliminaries}\label{sec:preliminaries}

Throughout this paper, let $\X$ be a finite state space with $N:=|\X|\ge2$, and let $\pi$ be a probability distribution with $\pi(x)>0$ for every $x\in\X$. For integers $a\le b$, we write $\llbracket a,b\rrbracket:=\{a,a+1,\ldots,b\}$ and $\llbracket b\rrbracket:=\llbracket1,b\rrbracket$. We set $\llbracket a,b\rrbracket:=\emptyset$ when $a>b$. All logarithms are natural. A transition matrix $P$ is $\pi$-stationary if $\pi P=\pi$. We denote the set of $\pi$-stationary transition matrices on $\X$ by
\[
\mathcal{S}(\pi):=\{P:P\text{ is a transition matrix on }\X,\ \pi P=\pi\}.
\]
A transition matrix $P$ is $\pi$-reversible if
\[
\pi(x)P(x,y)=\pi(y)P(y,x),\qquad x,y\in\X.
\]
An ergodic transition matrix is irreducible and aperiodic. We denote by $\Pi$ the transition matrix with $\Pi(x,y):=\pi(y)$.

For a probability distribution $\nu$ on a finite set, its Shannon entropy is
\[
H(\nu):=-\sum_z\nu(z)\log\nu(z).
\]
For $Q\in\mathcal S(\pi)$, we also write $H(Q)$ for the entropy rate of a $\pi$-stationary chain $Q$:
\begin{equation}\label{eq:entropy-rate-definition}
H(Q):=-\sum_{x,y\in\X}\pi(x)Q(x,y)\log Q(x,y).
\end{equation}
We use the convention $0\log0:=0$ in both definitions.

The space $\ell^2(\pi)$ is equipped with the inner product
\[
\inner{f}{g}_{\pi}:=\sum_{x\in\X}\pi(x)f(x)g(x),
\qquad \norm{f}_{\pi}^2:=\inner{f}{f}_{\pi}.
\]
We write $\ell^2_0(\pi):=\{f:\pi(f)=0\}$, where $\pi(f):=\sum_x\pi(x)f(x)$. A matrix $A$ acts on functions by $Af(x):=\sum_y A(x,y)f(y)$. Its $\ell^2(\pi)$-adjoint is
\[
A^*(x,y):=\frac{\pi(y)}{\pi(x)}A(y,x).
\]
In particular, a transition matrix $P$ is $\pi$-reversible if and only if $P=P^*$. For self-adjoint matrices, $A\preceq B$ means that $\inner{f}{Af}_{\pi}\le\inner{f}{Bf}_{\pi}$ for every $f\in\ell^2(\pi)$.

For real matrices $A,B$, define the weighted Frobenius inner product and norm by
\begin{align*}
\inner{A}{B}_{F,\pi}
&:=\Tr(A^*B)
=\sum_{x,y\in\X}\frac{\pi(x)}{\pi(y)}A(x,y)B(x,y),\\
\Fpi{A}^2&:=\inner{A}{A}_{F,\pi}.
\end{align*}
This is the Hilbert--Schmidt norm of $A$ acting on $\ell^2(\pi)$. We use $\opnorm{A}$ for the corresponding operator norm. If $P\in\mathcal{S}(\pi)$, Jensen's inequality and stationarity give
\[
\norm{Pf}_{\pi}^2
\le\sum_{x,y}\pi(x)P(x,y)f(y)^2
=\norm{f}_{\pi}^2.
\]
Since $P\Pi=\Pi P=\Pi$, the kernel $P$ preserves $\ell^2_0(\pi)$ and is a contraction on that subspace. It follows that
\begin{equation}
\Fpi{P-\Pi}^2\le N-1.
\end{equation}

\subsection{Bregman divergences}

Let $\phi:[0,\infty)\to\R$ be convex and continuous, and assume that it is continuously differentiable on $(0,\infty)$. The scalar Bregman divergence generated by $\phi$ is
\begin{equation}
B_\phi(u,v):=\phi(u)-\phi(v)-\phi'(v)(u-v),
\qquad u\ge0,\quad v>0.
\end{equation}
This construction goes back to \citet{Bregman1967}. At the boundary, we set $B_\phi(0,0):=0$. If the right derivative $\phi'_+(0)$ is finite, we use the same formula at $v=0$ with this derivative. If $\phi'_+(0)=-\infty$, we set $B_\phi(u,0):=+\infty$ for $u>0$. These conventions cover sparse transition matrices and the relative entropy generator.

For $P,Q\in\mathcal{S}(\pi)$, define
\begin{equation}\label{eq:kernel-Bregman}
\Dphi(P\|Q)
:=\sum_{x,y\in\X}\pi(x)\pi(y)
B_\phi\!\left(\frac{P(x,y)}{\pi(y)},\frac{Q(x,y)}{\pi(y)}\right).
\end{equation}
Thus the divergence is formed from the transition densities relative to $\pi$, with reference measure $\pi\otimes\pi$. We use these densities and weights throughout the paper. Convexity gives $\Dphi(P\|Q)\ge0$. If $\phi$ is strictly convex, equality holds if and only if $P=Q$. To see this, we note that strict convexity of $\phi$ gives $B_\phi(u,v)\ge0$, with equality
if and only if $u=v$, including the stated boundary conventions.
Since $\pi(x)\pi(y)>0$ for every $x,y\in\X$, the sum defining
$\Dphi(P\|Q)$ vanishes if and only if every summand vanishes.
This is equivalent to
\[
\frac{P(x,y)}{\pi(y)}=\frac{Q(x,y)}{\pi(y)}
\qquad\text{for every }x,y\in\X,
\]
which holds if and only if $P=Q$.

Adding an affine function to $\phi$ does not change $B_\phi$. We may therefore normalize the generator so that $\phi(1)=\phi'(1)=0$. Under this normalization,
\begin{equation*}
\Dphi(P\|\Pi)
=\sum_{x,y\in\X}\pi(x)\pi(y)
\phi\!\left(\frac{P(x,y)}{\pi(y)}\right).
\end{equation*}
Two choices will be used repeatedly. For $\phi(u)=u\log u-u+1$, with $0\log0=0$, the linear terms cancel and
\begin{equation*}
\Dphi(P\|Q)
=D_{\mathrm{KL}}^\pi(P\|Q)
:=\sum_{x,y}\pi(x)P(x,y)\log\frac{P(x,y)}{Q(x,y)}.
\end{equation*}
The summand is zero when $P(x,y)=0$, and it is infinite when $P(x,y)>0=Q(x,y)$. For $\phi(u)=(u-1)^2$, we have $B_\phi(u,v)=(u-v)^2$ and
\begin{equation*}
\Dphi(P\|Q)=\Fpi{P-Q}^2.
\end{equation*}
Other admissible generators include $\phi(u)=(u-1)^4$ and $\phi(u)=u^p$, $p>1$, with an affine normalization if desired.


\subsection{State-space partitions and Gibbs kernels}

Let $\mathcal{O}:=(\mathcal{O}_i)_{i=1}^k$ be a partition of $\X$ into $k$ nonempty blocks, and write $\overline\pi(i):=\pi(\mathcal{O}_i)$. The Gibbs kernel induced by $\mathcal{O}$ is
\begin{equation}
G(x,y):=\one_{\{y\in\mathcal{O}(x)\}}
\frac{\pi(y)}{\pi(\mathcal{O}(x))},
\end{equation}
where $\mathcal{O}(x)$ is the block containing $x$. One transition of $G$ resamples the state from $\pi$ conditional on its current block.

\begin{prop}\label{prop:G-projection}
The kernel $G\in\mathcal{S}(\pi)$ is reversible and satisfies
\[
G^2=G,\qquad G\Pi=\Pi G=\Pi.
\]
It is the orthogonal projection onto the subspace of functions that are constant on each block. In particular, $\rank(G)=k$, and $G-\Pi$ is an orthogonal projection of rank $k-1$ on $\ell^2_0(\pi)$.
\end{prop}

\begin{proof}
If $x,y\in\mathcal{O}_i$, then $\pi(x)G(x,y)=\pi(x)\pi(y)/\overline\pi(i)$, which is symmetric in $x,y$. If the states lie in different blocks, both sides of detailed balance are zero. The formula
\[
Gf(x)=\frac{1}{\overline\pi(i)}\sum_{y\in\mathcal{O}_i}\pi(y)f(y),
\qquad x\in\mathcal{O}_i,
\]
shows that $Gf$ is block-constant and that $G$ fixes every block-constant function. Thus $G^2=G$. Self-adjointness then identifies $G$ as the stated orthogonal projection. Since constants form a one-dimensional subspace of its range, the assertions about $G-\Pi$ follow.
\end{proof}

For a fixed partition, define the class of $G$-invariant kernels by
\begin{equation*}
\mathcal{G}
:=\{Q\in\mathcal{S}(\pi):GQG=Q\}.
\end{equation*}
The condition $GQG=Q$ is equivalent to $GQ=QG=Q$. It is also equivalent to requiring that $Q(x,y)/\pi(y)$ be constant on each block pair $\mathcal{O}_i\times\mathcal{O}_j$. The class $\mathcal{G}$ is convex, contains $\Pi$, and contains $GPG$ for every $P\in\mathcal{S}(\pi)$.

The projection chain of $P\in\mathcal{S}(\pi)$ on $\llbracket k\rrbracket$, as in \citet[Section~2]{JerrumEtAl2004}, is
\begin{equation*}
\overline P(i,j)
:=\frac{1}{\overline\pi(i)}
\sum_{\substack{x\in\mathcal{O}_i\\y\in\mathcal{O}_j}}
\pi(x)P(x,y).
\end{equation*}
It has stationary distribution $\overline\pi$. We write $H(\overline P)$ for its entropy rate, computed using $\overline\pi$ in \eqref{eq:entropy-rate-definition}. Direct multiplication gives
\begin{equation}\label{eq:GPG-block-form}
(GPG)(x,y)=\overline P(i,j)\frac{\pi(y)}{\overline\pi(j)},
\qquad x\in\mathcal{O}_i,\ y\in\mathcal{O}_j.
\end{equation}
This formula also verifies the block-density characterization of $\mathcal{G}$ above. We also see that
\begin{align}\label{eq:overlineGPGoverlineP}
	\overline{GPG} = \overline{P}.
\end{align}

\subsection{Group averaging}

Every finite partition can be realized as the orbit partition of a finite group. For example, choose a cyclic permutation on each block and take the product of the resulting cyclic groups. Conversely, the orbits of any finite group action on $\X$ form a partition and hence induce a Gibbs kernel.

Let $\Gamma$ be a finite group acting on $\X$, with orbit partition
$\mathcal O$. Each group element $g\in\Gamma$ acts as a permutation
of $\X$, and we write $gx:=g\cdot x$ for the image of $x$ under
this action. Thus, as $g$ ranges over $\Gamma$, the states $gx$
range over the orbit containing $x$, possibly with repetitions.
The Gibbs operator can be written as the weighted group average
\begin{equation}\label{eq:weighted-group-average}
	Gf(x)=
	\frac{\sum_{g\in\Gamma}\pi(gx)f(gx)}
	{\sum_{g\in\Gamma}\pi(gx)}.
\end{equation}
Indeed, every point in the orbit of $x$ occurs the same number of times in these sums. When $\pi(gx)=\pi(x)$ for all $g,x$, the weights cancel and
\[
G=\frac1{|\Gamma|}\sum_{g\in\Gamma}U_g,
\qquad U_gf(x):=f(gx).
\]
In this case,
\[
GPG=\frac1{|\Gamma|^2}\sum_{g,h\in\Gamma}U_gPU_h.
\]
For general $\pi$, formula~\eqref{eq:weighted-group-average} preserves the Gibbs interpretation without requiring invariance of the target under the group. We use the partition formulation below, so the particular group realizing a partition does not enter the results. This formulation agrees with the finite-state averaging constructions in \citet{ChoiWang2025,ChoiLimWang2025}.

\section{Bregman information projections}\label{sec:projection}

For $P\in\mathcal{S}(\pi)$, the kernel $GPG$ is the information projection of $P$ onto $\mathcal{G}$ for every strictly convex generator satisfying the assumptions of Section~\ref{sec:preliminaries}. The same kernel minimizes all these divergences because its density is the conditional mean of the density of $P$ on each block pair.

\begin{theorem}[Pythagorean identity]\label{thm:Bregman-projection}
Let $P\in\mathcal{S}(\pi)$ and let $G$ be induced by a fixed partition. Then, for every $Q\in\mathcal{G}$,
\begin{equation}\label{eq:Bregman-Pythagoras}
\Dphi(P\|Q)=\Dphi(P\|GPG)+\Dphi(GPG\|Q).
\end{equation}
The identity is understood in $[0,\infty]$, and $\Dphi(P\|GPG)<\infty$. Consequently,
\begin{equation*}
GPG\in\argmin_{Q\in\mathcal{G}}\Dphi(P\|Q).
\end{equation*}
If $\phi$ is strictly convex, this minimizer is unique.
\end{theorem}

\begin{proof}
By \eqref{eq:GPG-block-form}, the densities $(GPG)(x,y)/\pi(y)$ and $Q(x,y)/\pi(y)$ are constant on each block pair. Moreover, $P$ and $GPG$ have the same stationary flow between any two blocks:
\begin{equation}\label{eq:block-cancellation}
\sum_{x\in\mathcal{O}_i,\,y\in\mathcal{O}_j}
\pi(x)\bigl(P(x,y)-(GPG)(x,y)\bigr)= \overline{\pi}(i) \overline{P}(i,j) - \overline{\pi}(i) \overline{GPG}(i,j) = 0,
\end{equation}
where the last equality follows from \eqref{eq:overlineGPGoverlineP}. For $v,w>0$ and $u\ge0$, direct expansion gives the scalar identity
\begin{equation}\label{eq:three-point}
B_\phi(u,w)=B_\phi(u,v)+B_\phi(v,w)
+(u-v)\bigl(\phi'(v)-\phi'(w)\bigr).
\end{equation}
First consider a block pair on which $GPG$ and $Q$ are positive. Apply \eqref{eq:three-point} to the densities of $P,GPG,Q$, respectively, and sum with weights $\pi(x)\pi(y)$. The difference between the two sides of \eqref{eq:Bregman-Pythagoras} on this block pair is
\begin{align}\label{eq:block-Pythagoras}
\sum_{x\in\mathcal{O}_i,\,y\in\mathcal{O}_j}
&\pi(x)\bigl(P(x,y)-(GPG)(x,y)\bigr)
\left[
\phi'\!\left(\frac{(GPG)(x,y)}{\pi(y)}\right)
-\phi'\!\left(\frac{Q(x,y)}{\pi(y)}\right)
\right] \nonumber \\
&= \left[
\phi'\!\left(\frac{\overline{GPG}(i,j)}{\overline{\pi}(j)}\right)
-\phi'\!\left(\frac{\overline{Q}(i,j)}{\overline{\pi}(j)}\right)
\right]\sum_{x\in\mathcal{O}_i,\,y\in\mathcal{O}_j}
\pi(x)\bigl(P(x,y)-(GPG)(x,y)\bigr)
.
\end{align}
The bracketed term is constant on the block pair using \eqref{eq:GPG-block-form} and $GQG = Q$, so this sum vanishes by \eqref{eq:block-cancellation}.

If $GPG$ vanishes on a block pair, nonnegativity and \eqref{eq:block-cancellation} imply that $P$ vanishes there as well. The identity then holds directly on that block pair. If $Q$ vanishes and $GPG$ is positive, the same calculation uses $\phi'_+(0)$ when this derivative is finite. When $\phi'_+(0)=-\infty$, both $\Dphi(P\|Q)$ and $\Dphi(GPG\|Q)$ have an infinite contribution from that block pair. In all cases, $\Dphi(P\|GPG)$ remains finite.

Summing over the block pairs proves \eqref{eq:Bregman-Pythagoras}. Nonnegativity of $\Dphi(GPG\|Q)$ proves the minimizing property. Under strict convexity, this term vanishes only when $Q=GPG$, proving uniqueness.
\end{proof}

\begin{corollary}[Distance to stationarity]\label{cor:stationarity-Pythagoras}
For every $P\in\mathcal{S}(\pi)$,
\begin{equation}\label{eq:stationarity-Pythagoras}
\Dphi(P\|\Pi)
=\Dphi(P\|GPG)+\Dphi(GPG\|\Pi).
\end{equation}
In particular, $\Dphi(GPG\|\Pi)\le\Dphi(P\|\Pi)$.
\end{corollary}

\begin{proof}
Take $Q=\Pi\in\mathcal{G}$ in Theorem~\ref{thm:Bregman-projection}.
\end{proof}

For the relative entropy generator $\phi(u)=u\log u-u+1$, Theorem~\ref{thm:Bregman-projection} recovers the KL information projection and Pythagorean identities for Gibbs averaging in \citet{ChoiLimWang2025}, and the corresponding finite-state group-averaging identities in \citet{ChoiWang2025}. For the quadratic generator $\phi(u)=(u-1)^2$, it gives
\begin{equation}\label{eq:F-Pythagoras}
\Fpi{P-Q}^2
=\Fpi{P-GPG}^2+\Fpi{GPG-Q}^2,
\qquad Q\in\mathcal{G}.
\end{equation}
Thus the Hilbert--Schmidt projection interpretation is a special case of the same block-mean calculation. In this case the cancellation can also be expressed as the orthogonality
\[
\inner{P-GPG}{GPG-Q}_{F,\pi}=0.
\]
Neither Theorem~\ref{thm:Bregman-projection} nor these special cases require reversibility of $P$.

The projection chain gives a further reduction of the divergence after averaging. Let $\overline\Pi(i,j):=\overline\pi(j)$, and define $\mathcal D_\phi^{\overline\pi}$ by \eqref{eq:kernel-Bregman} on $\llbracket k\rrbracket$.

\begin{prop}[Reduction to the projection chain]\label{prop:divergence-reduction}
For every $P\in\mathcal{S}(\pi)$ and every positive integer $t$,
\begin{align}
(GPG)^t(x,y)
&=\overline P^{\,t}(i,j)\frac{\pi(y)}{\overline\pi(j)},
\quad x\in\mathcal{O}_i,\ y\in\mathcal{O}_j,\label{eq:power-lift}\\
\Dphi\bigl((GPG)^t\|\Pi\bigr)
&=\mathcal D_\phi^{\overline\pi}
\bigl(\overline P^{\,t}\|\overline\Pi\bigr).\notag
\end{align}
\end{prop}

\begin{proof}
	We first prove the formula for $(GPG)^t$ by induction on $t$.
	For $t=1$, \eqref{eq:power-lift} reduces to \eqref{eq:GPG-block-form}, which is true.
	
	Suppose that, for some positive integer $t$,
	\[
	(GPG)^t(x,y)
	=
	\overline{P}^{\,t}(i,j)\frac{\pi(y)}{\overline{\pi}(j)}
	\]
	for every block pair and every $x\in\mathcal{O}_i$,
	$y\in\mathcal{O}_j$. Then
	\begin{align*}
		(GPG)^{t+1}(x,y)
		&=
		\sum_{\ell=1}^k\sum_{z\in\mathcal{O}_\ell}
		(GPG)^t(x,z)(GPG)(z,y)\\
		&=
		\sum_{\ell=1}^k\sum_{z\in\mathcal{O}_\ell}
		\overline{P}^{\,t}(i,\ell)
		\frac{\pi(z)}{\overline{\pi}(\ell)}
		\overline{P}(\ell,j)
		\frac{\pi(y)}{\overline{\pi}(j)}\\
		&=
		\frac{\pi(y)}{\overline{\pi}(j)}
		\sum_{\ell=1}^k
		\overline{P}^{\,t}(i,\ell)\overline{P}(\ell,j)
		\left(
		\frac{\sum_{z\in\mathcal{O}_\ell}\pi(z)}
		{\overline{\pi}(\ell)}
		\right)\\
		&=
		\frac{\pi(y)}{\overline{\pi}(j)}
		\sum_{\ell=1}^k
		\overline{P}^{\,t}(i,\ell)\overline{P}(\ell,j)\\
		&=
		\overline{P}^{\,t+1}(i,j)
		\frac{\pi(y)}{\overline{\pi}(j)}.
	\end{align*}
	The second equality uses the induction hypothesis and the case
	$t=1$. The fourth equality follows from
	$\sum_{z\in\mathcal{O}_\ell}\pi(z)=\overline{\pi}(\ell)$.
	This completes the induction.
	
	Now fix any positive integer $t$.
	Since $\Pi(x,y)=\pi(y)$, the definition of the Bregman divergence
	and the formula just proved yield
	\begin{align*}
		\mathcal{D}_{\phi}^{\pi}\bigl((GPG)^t\|\Pi\bigr)
		&=
		\sum_{i,j=1}^k
		\sum_{\substack{x\in\mathcal{O}_i\\y\in\mathcal{O}_j}}
		\pi(x)\pi(y)
		B_{\phi}\!\left(
		\frac{(GPG)^t(x,y)}{\pi(y)},1
		\right)\\
		&=
		\sum_{i,j=1}^k
		\left(\sum_{x\in\mathcal{O}_i}\pi(x)\right)
		\left(\sum_{y\in\mathcal{O}_j}\pi(y)\right)
		B_{\phi}\!\left(
		\frac{\overline{P}^{\,t}(i,j)}{\overline{\pi}(j)},1
		\right)\\
		&=
		\sum_{i,j=1}^k
		\overline{\pi}(i)\overline{\pi}(j)
		B_{\phi}\!\left(
		\frac{\overline{P}^{\,t}(i,j)}{\overline{\pi}(j)},
		\frac{\overline{\Pi}(i,j)}{\overline{\pi}(j)}
		\right)\\
		&=
		\mathcal{D}_{\phi}^{\overline{\pi}}
		\bigl(\overline{P}^{\,t}\|\overline{\Pi}\bigr).
	\end{align*}
	Here the Bregman term is constant on each block pair, and
	$\overline{\Pi}(i,j)=\overline{\pi}(j)$.
\end{proof}

\subsection{Dual projection and maximum entropy rate}\label{subsec:dual-projection}

Fix $P\in\mathcal{S}(\pi)$ and the partition inducing $G$. The stationary flows between blocks determine $GPG$ through \eqref{eq:GPG-block-form}. Consider the class of kernels with these prescribed flows:
\begin{equation*}
\mathcal{F}(P):=\{Q\in\mathcal{S}(\pi):\overline Q=\overline P\}.
\end{equation*}
Equivalently, $Q\in\mathcal{F}(P)$ if and only if
\[
\sum_{x\in\mathcal{O}_i,\,y\in\mathcal{O}_j}\pi(x)Q(x,y)
=
\sum_{x\in\mathcal{O}_i,\,y\in\mathcal{O}_j}\pi(x)P(x,y),
\qquad i,j\in\llbracket k\rrbracket.
\]
This condition fixes the joint distribution of two consecutive block labels when the chain starts from $\pi$. 

\begin{prop}[Dual information projection]\label{prop:dual-projection}
For every $Q\in\mathcal{F}(P)$ and $T\in\mathcal{G}$,
\begin{equation}\label{eq:dual-general-Pythagoras}
\Dphi(Q\|T)
=
\Dphi(Q\|GPG)+\Dphi(GPG\|T).
\end{equation}
In particular,
\begin{equation}\label{eq:dual-Pythagoras}
\Dphi(Q\|\Pi)
=
\Dphi(Q\|GPG)+\Dphi(GPG\|\Pi),
\qquad Q\in\mathcal{F}(P),
\end{equation}
and
\begin{equation}\label{eq:dual-minimizer}
GPG\in\argmin_{Q\in\mathcal{F}(P)}\Dphi(Q\|\Pi).
\end{equation}
The minimizer is unique if $\phi$ is strictly convex. Moreover,
$\mathcal{F}(P)\cap\mathcal{G}=\{GPG\}$.
\end{prop}

\begin{proof}
If $Q\in\mathcal{F}(P)$, then \eqref{eq:GPG-block-form} gives $GQG=GPG$. Applying Theorem~\ref{thm:Bregman-projection} to $Q$ with competitor $T$ proves \eqref{eq:dual-general-Pythagoras}. Taking $T=\Pi$ gives \eqref{eq:dual-Pythagoras}. The kernel $GPG$ belongs to $\mathcal{F}(P)$ by \eqref{eq:block-cancellation}. Nonnegativity of $\Dphi(Q\|GPG)$ proves \eqref{eq:dual-minimizer}, and strict convexity makes equality possible only for $Q=GPG$. Finally, if $Q$ also belongs to $\mathcal{G}$, then $Q=GQG=GPG$, proving the intersection statement.
\end{proof}

The two projection problems have different variable arguments. Theorem~\ref{thm:Bregman-projection} minimizes $\Dphi(P\|Q)$ over $Q\in\mathcal{G}$ in the second argument, whereas Proposition~\ref{prop:dual-projection} minimizes $\Dphi(Q\|\Pi)$ over $Q\in\mathcal{F}(P)$ in the first argument. Identity~\eqref{eq:dual-general-Pythagoras} relates these two families through their common element $GPG$, as illustrated in Figure~\ref{fig:dual-Pythagoras}. In the quadratic case, it is an orthogonal decomposition in the weighted Frobenius inner product. For a general generator, the block cancellation in \eqref{eq:block-Pythagoras} gives the same additive identity.

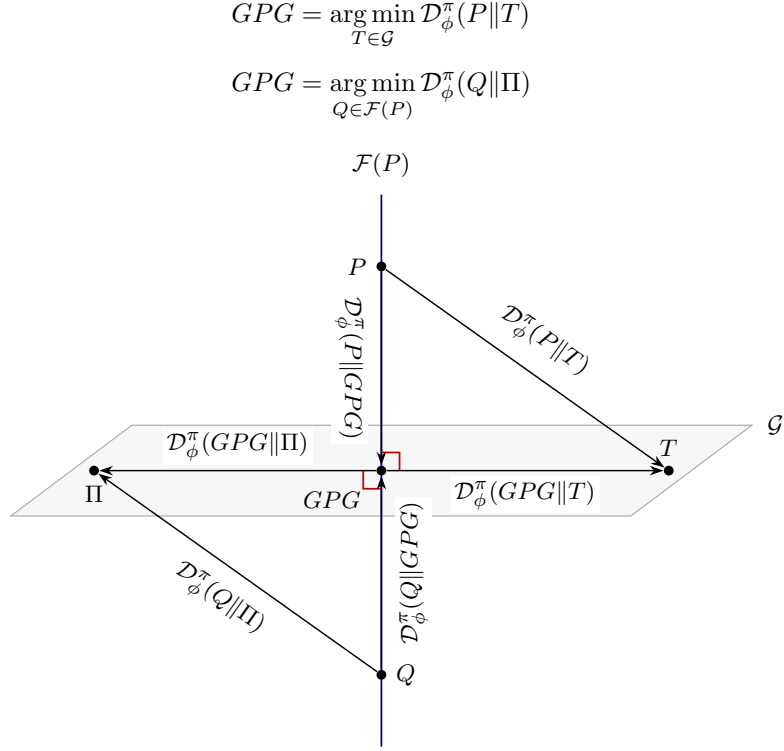
\begin{figure}[!t]
\centering
\begin{tikzpicture}[
  font=\small,
  divergence/.style={-{Stealth[length=1.8mm,width=1.25mm]},
    line width=0.55pt,shorten <=2pt,shorten >=2pt},
  divergence label/.style={fill=white,inner sep=1.5pt},
  vertical family/.style={draw=blue!55!black,line width=0.65pt},
  horizontal plane/.style={draw=gray!65,fill=gray!8,
    fill opacity=0.78,line width=0.45pt},
  point/.style={circle,fill=black,inner sep=1.3pt}
]
  \node[anchor=south,inner sep=0pt,font=\normalsize] at (0,4.6) {$
    \begin{aligned}
      GPG&=\argmin_{T\in\mathcal G}\Dphi(P\|T)\\[0.35em]
      GPG&=\argmin_{Q\in\mathcal F(P)}\Dphi(Q\|\Pi)
    \end{aligned}
  $};
  \coordinate (projection) at (0,0);
  \coordinate (P) at (0,2.7);
  \coordinate (Q) at (0,-2.7);
  \coordinate (T) at (3.8,0);
  \coordinate (Pi) at (-3.8,0);
  \path[horizontal plane]
    (-4.9,-0.6) -- (3.3,-0.6) -- (4.9,0.6)
    -- (-3.3,0.6) -- cycle;
  \draw[vertical family] (0,-3.65) -- (0,3.65);
  \node[above] at (0,3.76) {$\mathcal F(P)$};
  \node[right] at (4.98,0.6) {$\mathcal G$};
  \draw[red!75!black,line width=0.65pt]
    (0,0.24) -- (0.24,0.24) -- (0.24,0);
  \draw[red!75!black,line width=0.65pt]
    (0,-0.24) -- (-0.24,-0.24) -- (-0.24,0);
  \draw[divergence] (P) --
    node[divergence label,midway,sloped,below=5pt] {$\Dphi(P\|GPG)$} (projection);
  \draw[divergence] (projection) --
    node[divergence label,midway,below=2pt] {$\Dphi(GPG\|T)$} (T);
  \draw[divergence] (P) --
    node[divergence label,midway,sloped,above=5pt] {$\Dphi(P\|T)$} (T);
  \draw[divergence] (Q) --
    node[divergence label,midway,sloped,below=5pt] {$\Dphi(Q\|GPG)$} (projection);
  \draw[divergence] (projection) --
    node[divergence label,midway,above=2pt] {$\Dphi(GPG\|\Pi)$} (Pi);
  \draw[divergence] (Q) --
    node[divergence label,midway,sloped,below=5pt] {$\Dphi(Q\|\Pi)$} (Pi);
  \node[point,label={left:$P$}] at (P) {};
  \node[point,label={right:$Q$}] at (Q) {};
  \node[point,label={above:$T$}] at (T) {};
  \node[point,label={below:$\Pi$}] at (Pi) {};
  \node[point,label={[label distance=4pt]below left:$GPG$}]
    at (projection) {};
\end{tikzpicture}
\caption{Pythagorean identities at $GPG$. The paths $P\to GPG\to T$ and $Q\to GPG\to\Pi$ illustrate Theorem~\ref{thm:Bregman-projection} and Proposition~\ref{prop:dual-projection}, respectively, for $T\in\mathcal G$ and $Q\in\mathcal F(P)$. Each arrow indicates divergence in its direction; each diagonal divergence equals the sum along the corresponding path. The vertical line and horizontal plane schematically depict $\mathcal F(P)$ and $\mathcal G$, with their unique intersection at $GPG$; the depicted dimensions are illustrative. The red right-angle marks represent the Bregman Pythagorean identities. Both minimization formulas assume strictly convex $\phi$.}
\label{fig:dual-Pythagoras}
\end{figure}

\begin{samepage}
\begin{corollary}[Maximum entropy rate under prescribed block flows]\label{cor:maximum-entropy}
The kernel $GPG$ uniquely maximizes $H(Q)$ over $Q\in\mathcal{F}(P)$. More precisely,
\begin{equation}\label{eq:entropy-rate-gap}
H(GPG)-H(Q)=D_{\mathrm{KL}}^\pi(Q\|GPG),
\qquad Q\in\mathcal{F}(P).
\end{equation}
Its entropy rate is
\begin{align}
H(GPG)
&=H(\overline P)+
\sum_{j=1}^k\overline\pi(j)
H\bigl(\pi(\,\cdot\,\mid\mathcal{O}_j)\bigr)\label{eq:entropy-rate-decomposition}\\
&=H(\pi)-D_{\mathrm{KL}}^{\overline\pi}
(\overline P\|\overline\Pi).\label{eq:entropy-rate-coarse-KL}
\end{align}
These statements require neither reversibility nor irreducibility.
\end{corollary}
\end{samepage}

\begin{proof}
Stationarity gives
\[
D_{\mathrm{KL}}^\pi(Q\|\Pi)
=
\sum_{x,y}\pi(x)Q(x,y)\log Q(x,y)
-\sum_y\pi(y)\log\pi(y)
=
H(\pi)-H(Q).
\]
Substituting this identity into \eqref{eq:dual-Pythagoras} proves \eqref{eq:entropy-rate-gap}. The KL divergence vanishes only when $Q=GPG$, giving uniqueness.

To obtain \eqref{eq:entropy-rate-decomposition}, substitute \eqref{eq:GPG-block-form} into \eqref{eq:entropy-rate-definition} and split the logarithm into $\log\overline P(i,j)$ and $\log(\pi(y)/\overline\pi(j))$ on block pairs with positive flow. The first contribution is $H(\overline P)$. The second is
\[
\sum_{i,j}\overline\pi(i)\overline P(i,j)
H\bigl(\pi(\,\cdot\,\mid\mathcal{O}_j)\bigr)
=
\sum_j\overline\pi(j)
H\bigl(\pi(\,\cdot\,\mid\mathcal{O}_j)\bigr).
\]
Zero-flow block pairs contribute zero. Finally, Proposition~\ref{prop:divergence-reduction} at $t=1$ identifies
$D_{\mathrm{KL}}^\pi(GPG\|\Pi)$ with
$D_{\mathrm{KL}}^{\overline\pi}(\overline P\|\overline\Pi)$, proving \eqref{eq:entropy-rate-coarse-KL}.
\end{proof}

The entropy gain has a direct interpretation in terms of dependence.
For $Q\in\mathcal{F}(P)$, define the joint law
\[
\mathbb{P}_Q(X_0=x,X_1=y):=\pi(x)Q(x,y),
\qquad x,y\in\X.
\]
For finite-valued random variables $U,V$ under this law, their
mutual information is
\[
I_Q(U;V)
:=
\sum_{u,v}
\mathbb{P}_Q(U=u,V=v)
\log
\frac{\mathbb{P}_Q(U=u,V=v)}
{\mathbb{P}_Q(U=u)\mathbb{P}_Q(V=v)},
\]
where terms with zero joint probability are taken to be zero.
Thus the subscript $Q$ indicates that all probabilities are
computed under $\mathbb{P}_Q$.

Since $Q\in\mathcal{S}(\pi)$, both $X_0$ and $X_1$ have marginal
distribution $\pi$. Consequently,
\[
I_Q(X_0;X_1)
=
\sum_{x,y\in\X}\pi(x)Q(x,y)
\log\frac{Q(x,y)}{\pi(y)}
=
D_{\mathrm{KL}}^\pi(Q\|\Pi).
\]
For $t\in\{0,1\}$, let $B_t:=i$ when $X_t\in\mathcal{O}_i$.
Both block labels have marginal distribution $\overline\pi$.
Moreover, the condition $Q\in\mathcal{F}(P)$ gives
\[
\mathbb{P}_Q(B_0=i,B_1=j)
=
\sum_{\substack{x\in\mathcal{O}_i\\y\in\mathcal{O}_j}}
\pi(x)Q(x,y)
=
\overline\pi(i)\overline P(i,j).
\]
Therefore,
\[
I_Q(B_0;B_1)
=
\sum_{i,j=1}^k
\overline\pi(i)\overline P(i,j)
\log\frac{\overline P(i,j)}{\overline\pi(j)}
=
D_{\mathrm{KL}}^{\overline\pi}
(\overline P\|\overline\Pi).
\]
Consequently,
\begin{equation*}
D_{\mathrm{KL}}^\pi(Q\|GPG)
=
I_Q(X_0;X_1)-I_Q(B_0;B_1).
\end{equation*}
Thus $GPG$ retains exactly the dependence prescribed by the block flows. Its stationary joint law factors as, using \eqref{eq:GPG-block-form} again,
\begin{equation*}
\pi(x)(GPG)(x,y)
=
\overline\pi(i)\overline P(i,j)
\frac{\pi(x)}{\overline\pi(i)}
\frac{\pi(y)}{\overline\pi(j)},
\qquad x\in\mathcal{O}_i,\quad y\in\mathcal{O}_j.
\end{equation*}
Conditional on a pair of block labels with positive probability, the two states are therefore independent draws from the corresponding conditional distributions of $\pi$. In particular, taking $Q=P$ in \eqref{eq:entropy-rate-gap} identifies the entropy gain $H(GPG)-H(P)$ with the information projection distance $D_{\mathrm{KL}}^\pi(P\|GPG)$.

\begin{remark}[Choosing a partition by entropy rate]\label{rem:entropy-partition}
When $G$ varies over kernels induced by $k$ nonempty blocks, \eqref{eq:entropy-rate-coarse-KL} shows that maximizing $H(GPG)$ is equivalent to minimizing
\[
D_{\mathrm{KL}}^{\overline\pi}(\overline P\|\overline\Pi)
=
\sum_{i,j=1}^k
\overline\pi(i)\overline P(i,j)
\log\frac{\overline P(i,j)}{\overline\pi(j)}.
\]
Both $\overline\pi$ and $\overline P$ depend on the partition. Since by Proposition \ref{prop:divergence-reduction} we have
$H(GPG)=H(\pi)+H(\overline P)-H(\overline\pi)$, the term $H(\overline\pi)$ must be included when block masses vary. The upper bound $H(GPG)\le H(\pi)$ is attained precisely when $\overline P=\overline\Pi$, equivalently $GPG=\Pi$.

\end{remark}

\subsection{Coarsening and hierarchical decompositions}\label{subsec:coarsening}

Coarsening a partition averages over larger blocks. The resulting Gibbs projections are nested, so the Pythagorean identity separates the divergence removed by each coarsening step from the divergence that remains.

\begin{prop}[Monotonicity under coarsening]\label{prop:coarsening}
	Let $P\in\mathcal S(\pi)$, and let $\mathcal O'$ be coarser than
	$\mathcal O$, meaning that every block of $\mathcal O$ is contained
	in a block of $\mathcal O'$. Denote their Gibbs kernels by $G'$ and
	$G$, respectively. Then
	\begin{equation}\label{eq:nested-Gibbs}
		G'G=GG'=G',
		\qquad
		G'PG'=G'(GPG)G'.
	\end{equation}
	
	Let $\phi:[0,\infty)\to\mathbb R$ be convex and continuous on
	$[0,\infty)$, and continuously differentiable on $(0,\infty)$.
	With the convention $B_\phi(0,0)=0$, we have
	\begin{equation}\label{eq:telescoping}
		\Dphi(P\|\Pi)
		=
		\Dphi(P\|GPG)
		+\Dphi(GPG\|G'PG')
		+\Dphi(G'PG'\|\Pi).
	\end{equation}
	In particular,
	\begin{equation}\label{eq:coarsening-divergence}
		\Dphi(G'PG'\|\Pi)
		\le
		\Dphi(GPG\|\Pi)
		\le
		\Dphi(P\|\Pi).
	\end{equation}
\end{prop}

\begin{proof}
Every function constant on the blocks of $\mathcal O'$ is constant on the blocks of $\mathcal O$. Thus the range of $G'$ is contained in the range of $G$. Since both kernels are orthogonal projections, $G'G=GG'=G'$, which also gives $G'(GPG)G'=G'PG'$.

Apply Corollary~\ref{cor:stationarity-Pythagoras} first to $P$ with Gibbs kernel $G$ and then to $GPG$ with Gibbs kernel $G'$. Using \eqref{eq:nested-Gibbs}, we obtain
\begin{align*}
\Dphi(P\|\Pi)
&=\Dphi(P\|GPG)+\Dphi(GPG\|\Pi),\\
\Dphi(GPG\|\Pi)
&=\Dphi(GPG\|G'PG')+\Dphi(G'PG'\|\Pi).
\end{align*}
Substitution proves \eqref{eq:telescoping}, and nonnegativity gives \eqref{eq:coarsening-divergence}.
\end{proof}

For reversible ergodic kernels, Corollary~\ref{cor:coarsening-all-time} below gives the corresponding squared-Frobenius comparison at every positive integer time.

The invariant families are nested as well. Define
\[
\mathcal G':=\{Q\in\mathcal S(\pi):G'QG'=Q\}.
\]
For $Q\in\mathcal G'$, identity~\eqref{eq:nested-Gibbs} gives
$GQG=GG'QG'G=G'QG'=Q$, so $\mathcal G'\subseteq\mathcal G$.
Figure~\ref{fig:coarsening} shows both the successive projections and the direct projection of $P$ onto $\mathcal G'$.

\begin{figure}[!htbp]
\centering
\begin{tikzpicture}[
  font=\small,
  divergence/.style={-{Stealth[length=1.8mm,width=1.25mm]},
    line width=0.55pt,shorten <=2pt,shorten >=2pt},
  direct/.style={divergence,draw=blue!60!black,line width=0.75pt},
  divergence label/.style={fill=white,inner sep=1.5pt},
  point/.style={circle,fill=black,inner sep=1.3pt}
]
  \node[anchor=south,inner sep=0pt,font=\normalsize] at (0.55,4.15) {$
    \begin{aligned}
      GPG&=\argmin_{T\in\mathcal G}\Dphi(P\|T)\\[0.35em]
      G'PG'&=\argmin_{T\in\mathcal G'}\Dphi(P\|T)\\[0.35em]
      G'PG'&=\argmin_{T\in\mathcal G'}\Dphi(GPG\|T)
    \end{aligned}
  $};
  \coordinate (P) at (-2.8,3.2);
  \coordinate (fine) at (-2.8,-1.6);
  \coordinate (coarse) at (0,0);
  \coordinate (Pi) at (4,0);
  \path[draw=gray!65,fill=blue!6,line width=0.45pt]
    (-4.675,-2.1)
      .. controls (-3,-2.45) and (-0.7,-2.45) .. (0.925,-2.1)
      .. controls (2.7,-1.3) and (4.35,-0.25) .. (5.825,0.7)
      .. controls (3.95,1.1) and (1.95,1.15) .. (0.225,0.7)
      .. controls (-1.3,-0.1) and (-3.3,-1.2) .. cycle;
  \node[below] at (-3.6,-2.34) {$\mathcal G$};
  \draw[gray!70,line width=0.8pt]
    (-0.75,-0.35) .. controls (-0.5,-0.27) and (-0.25,0) .. (coarse)
    .. controls (1.15,0) and (2.95,0.55) .. (Pi)
    .. controls (4.10,-0.052) and (4.20,-0.10) .. (4.30,-0.14);
  \node[right] at (4.65,-0.35) {$\mathcal G'\subseteq\mathcal G$};
  \draw[divergence] (P)
    .. controls (-3.55,1.85) and (-3.4,-0.8) ..
    node[divergence label,midway,sloped,below=5pt]
      {$\Dphi(P\|GPG)$} (fine);
  \draw[divergence] (fine)
    .. controls (-2.2,-1.15) and (0,-0.8) ..
    node[divergence label,midway,sloped,below=7pt]
      {$\Dphi(GPG\|G'PG')$} (coarse);
  \draw[direct] (coarse)
    .. controls (1.15,0) and (2.95,0.55) ..
    node[divergence label,text=blue!60!black,midway,below=6pt]
      {$\Dphi(G'PG'\|\Pi)$} (Pi);
  \draw[direct] (P)
    .. controls (-0.95,2.9) and (0,1.35) ..
    node[divergence label,text=blue!60!black,midway,sloped,below=5pt]
      {$\Dphi(P\|G'PG')$} (coarse);
  \draw[divergence] (P)
    .. controls (-0.3,3.6) and (2.9,2) ..
    node[divergence label,midway,sloped,above=5pt]
      {$\Dphi(P\|\Pi)$} (Pi);
  \draw[red!80!black,line width=0.75pt]
    (-2.98,-1.36) -- (-2.74,-1.18) -- (-2.56,-1.42);
  \draw[red!80!black,line width=0.75pt]
    (0,0.25) -- (0.25,0.25) -- (0.25,0);
  \draw[red!80!black,line width=0.75pt]
    (0,-0.25) -- (0.25,-0.25) -- (0.25,0);
  \node[point,label={above left:$P$}] at (P) {};
  \node[point,label={left:$GPG$}] at (fine) {};
  \node[point,label={[label distance=8pt]above left:$G'PG'$}] at (coarse) {};
  \node[point,label={below:$\Pi$}] at (Pi) {};
\end{tikzpicture}
\caption{Coarsening and the Pythagorean decomposition. The surface and curve schematically depict the nested invariant families $\mathcal G$ and $\mathcal G'$. The three-step path $P\to GPG\to G'PG'\to\Pi$ carries the nonnegative summands in \eqref{eq:telescoping}. The blue path $P\to G'PG'\to\Pi$ shows the direct projection onto $\mathcal G'$. The divergence on the direct arrow from $P$ to $\Pi$ equals the sum along either path. The curved arrows are schematic, and their labels give endpoint divergences rather than arc lengths. The red right-angle marks represent the Bregman Pythagorean identities. The three minimization formulas are stated for strictly convex $\phi$.}
\label{fig:coarsening}
\end{figure}
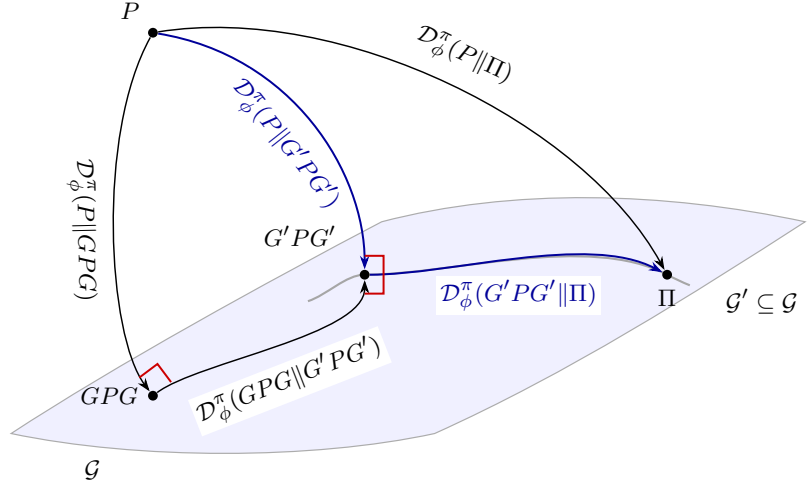

More generally, let $\mathcal O_0,\ldots,\mathcal O_{m+1}$ be successively coarser partitions, with $\mathcal O_0$ the singleton partition and $\mathcal O_{m+1}=\{\X\}$. Let $G_i$ be the Gibbs kernel of $\mathcal O_i$. Then $G_0=I$, $G_{m+1}=\Pi$, and repeated application of \eqref{eq:telescoping} gives
\begin{equation*}
\Dphi(P\|\Pi)
=\sum_{i=0}^{m}\Dphi(G_iPG_i\|G_{i+1}PG_{i+1}).
\end{equation*}
Each nonnegative summand measures the divergence removed at one level of the hierarchy. For strictly convex $\phi$, that summand vanishes exactly when $G_iPG_i=G_{i+1}PG_{i+1}$. 


\section{Frobenius comparison at every time}\label{sec:all-time}

The Pythagorean identity compares $P$ and $GPG$ after one transition. For reversible $P$, singular-value compression gives a comparison at every positive integer time. Reversibility is used to express the squared Frobenius norm of a power as a sum of powers of singular values.

\begin{theorem}\label{thm:all-time}
Let $P\in\mathcal{S}(\pi)$ be ergodic and reversible, and let $G$ be induced by a partition into $k$ nonempty blocks, where $k\in\llbracket2,N\rrbracket$. Denote the singular values of $P$ restricted to $\ell^2_0(\pi)$ by
\[
1>\sigma_1\ge\sigma_2\ge\cdots\ge\sigma_{N-1}\ge0.
\]
Then $GPG$ is ergodic and reversible. For every positive integer $t$,
\begin{equation}\label{eq:all-time}
\Fpi{(GPG)^t-\Pi}^2
\le\sum_{j=1}^{k-1}\sigma_j^{2t}
\le\sum_{j=1}^{N-1}\sigma_j^{2t}
=\Fpi{P^t-\Pi}^2.
\end{equation}
\end{theorem}

\begin{proof}
The matrices $P-\Pi$ and $GPG-\Pi=G(P-\Pi)G$ are self-adjoint. Since $G-\Pi$ has rank $k-1$,
\[
\rank(GPG-\Pi)\le k-1.
\]
We compare their singular values directly. For a matrix $A$ acting on $\ell^2(\pi)$, write $s_1(A)\ge\cdots\ge s_N(A)\ge0$ for its singular values, counted with multiplicity and including zeros. The spectral-norm Eckart--Young theorem \citep[Theorem~2.4.8, p.~79]{GolubVanLoan2013}, applied in an orthonormal basis of $\ell^2(\pi)$, gives the rank-approximation formula
\begin{equation}\label{eq:singular-rank-approximation}
s_j(A)=\inf_{\rank(C)<j}\opnorm{A-C},
\qquad j\in\llbracket1,N\rrbracket.
\end{equation}
To prove it, truncate the singular-value decomposition after $j-1$ terms. The resulting matrix has rank less than $j$ and approximation error $s_j(A)$. Conversely, if $\rank(C)<j$, there is a unit vector $f$ in the span of the first $j$ right singular vectors such that $Cf=0$. Thus
\[
\opnorm{A-C}\ge\norm{(A-C)f}_{\pi}
=\norm{Af}_{\pi}\ge s_j(A),
\]
which proves the reverse inequality in \eqref{eq:singular-rank-approximation}.

Since $\rank(GCG)<j$ whenever $\rank(C)<j$, and $G$ is an orthogonal projection, \eqref{eq:singular-rank-approximation} gives
\begin{align*}
s_j(GPG-\Pi)
&\le\inf_{\rank(C)<j}\opnorm{G(P-\Pi)G-GCG}\\
&\le\inf_{\rank(C)<j}\opnorm{P-\Pi-C}
=s_j(P-\Pi).
\end{align*}
We have therefore proved the singular-value compression inequality
\begin{equation}\label{eq:singular-compression}
s_j(GPG-\Pi)\le s_j(P-\Pi),
\qquad j\in\llbracket1,N\rrbracket.
\end{equation}
For $j\in\llbracket1,N-1\rrbracket$, we have $s_j(P-\Pi)=\sigma_j$, while $s_N(P-\Pi)=0$.

For $K\in\mathcal{S}(\pi)$, $K^t-\Pi=(K-\Pi)^t$ when $t\ge1$. Self-adjointness therefore gives
\begin{align*}
\Fpi{(GPG)^t-\Pi}^2
&=\sum_{j=1}^{k-1}s_j(GPG-\Pi)^{2t}\\
&\le\sum_{j=1}^{k-1}\sigma_j^{2t}
\le\sum_{j=1}^{N-1}\sigma_j^{2t}
=\Fpi{P^t-\Pi}^2.
\end{align*}
Finally,
\[
\opnorm{GPG-\Pi}\le\opnorm{P-\Pi}=\sigma_1<1.
\]
Hence $(GPG)^t\to\Pi$. Since every entry of $\Pi$ is positive, a sufficiently large power of $GPG$ has all entries positive, proving ergodicity. Reversibility follows from $(GPG)^*=GPG$.
\end{proof}

\begin{corollary}[Coarsening at every time]\label{cor:coarsening-all-time}
Let $P\in\mathcal S(\pi)$ be reversible and ergodic. Let $\mathcal O'$ be coarser than $\mathcal O$, with Gibbs kernels $G'$ and $G$, respectively. Then, for every positive integer $t$,
\begin{equation*}
\Fpi{(G'PG')^t-\Pi}^2
\le\Fpi{(GPG)^t-\Pi}^2
\le\Fpi{P^t-\Pi}^2.
\end{equation*}
\end{corollary}

\begin{proof}
The second inequality follows from Theorem~\ref{thm:all-time} when $G\ne\Pi$ and is immediate when $G=\Pi$. If $G'=\Pi$, the first inequality is also immediate. Otherwise, both partitions have at least two blocks, and Theorem~\ref{thm:all-time} shows that $GPG$ is reversible and ergodic. Apply that theorem to $GPG$ with Gibbs kernel $G'$, using $G'(GPG)G'=G'PG'$ from \eqref{eq:nested-Gibbs}, to obtain the first inequality.
\end{proof}

\begin{corollary}[Strict improvement and equality]\label{cor:strict-F-improvement}
Under the assumptions of Theorem~\ref{thm:all-time}, equality
\[
\Fpi{(GPG)^t-\Pi}^2=\Fpi{P^t-\Pi}^2
\]
holds for some positive integer $t$ if and only if $P=GPG$. Consequently, if $P\ne GPG$, then
\begin{equation}\label{eq:strict-F-improvement}
\Fpi{(GPG)^t-\Pi}^2<\Fpi{P^t-\Pi}^2
\qquad\text{for every positive integer }t.
\end{equation}
\end{corollary}

\begin{proof}
Suppose equality holds at a positive integer $t$. Self-adjointness and the identities $P^t-\Pi=(P-\Pi)^t$ and $(GPG)^t-\Pi=(GPG-\Pi)^t$ give
\[
0=\Fpi{P^t-\Pi}^2-\Fpi{(GPG)^t-\Pi}^2
=\sum_{j=1}^{N}\bigl(s_j(P-\Pi)^{2t}-s_j(GPG-\Pi)^{2t}\bigr).
\]
Every summand is nonnegative by \eqref{eq:singular-compression}. Since their sum is zero, each summand is zero. The function $u\mapsto u^{2t}$ is strictly increasing on $[0,\infty)$, so
\[
s_j(P-\Pi)=s_j(GPG-\Pi),
\qquad j\in\llbracket1,N\rrbracket.
\]
Summing their squares yields $\Fpi{P-\Pi}^2=\Fpi{GPG-\Pi}^2$. Applying the quadratic Pythagorean identity \eqref{eq:F-Pythagoras} with $Q=\Pi$ now gives
\[
\Fpi{P-GPG}^2
=\Fpi{P-\Pi}^2-\Fpi{GPG-\Pi}^2=0,
\]
and hence $P=GPG$. Conversely, if $P=GPG$, equality holds at every positive integer time. Combining this characterization with Theorem~\ref{thm:all-time} proves \eqref{eq:strict-F-improvement}.
\end{proof}

Theorem~\ref{thm:all-time} allows negative eigenvalues of $P$ and imposes no compatibility condition between $P$ and the partition. In the two-block case, it reduces to
\[
\Fpi{(GPG)^t-\Pi}^2\le\sigma_1^{2t}\le\Fpi{P^t-\Pi}^2.
\]
For $k=N$, every block is a singleton, $G=I$, and equality holds. The case $k=1$ follows directly from $G=\Pi$ and $GPG=\Pi$. The proof compares the compressed operator and its powers without requiring $P$ and $G$ to commute.

For an ergodic $K\in\mathcal{S}(\pi)$ and $\varepsilon>0$, define the squared-Frobenius mixing time by
\begin{equation*}
t_{F,\pi}(K,\varepsilon)
:=\min\{t\in\mathbb Z_{\ge0}:\Fpi{K^t-\Pi}^2\le\varepsilon\}.
\end{equation*}
The squared distance is nonincreasing in $t$, since
\[
\Fpi{K^{t+1}-\Pi}
=\Fpi{K(K^t-\Pi)}\le\opnorm{K}\Fpi{K^t-\Pi}
\le\Fpi{K^t-\Pi}.
\]
At $t=0$, this distance is $\Fpi{I-\Pi}^2=N-1$.

\begin{corollary}[Mixing-time comparison]\label{cor:F-mixing-comparison}
Under the assumptions of Theorem~\ref{thm:all-time},
\begin{equation*}
t_{F,\pi}(GPG,\varepsilon)\le t_{F,\pi}(P,\varepsilon),
\qquad \varepsilon>0.
\end{equation*}
\end{corollary}

\begin{proof}
At $t=0$, the two squared distances agree. At every positive integer time, apply \eqref{eq:all-time}.
\end{proof}

The strict distance comparison in Corollary~\ref{cor:strict-F-improvement} need not make the mixing-time inequality strict, since both chains may first meet a given tolerance at the same integer time.

Thus every valid block partition gives a sampler that is no worse than a reversible baseline in the number of transitions needed to meet a squared-Frobenius tolerance. This statement does not require optimizing the partition. Its computational interpretation depends on the cost of sampling from the conditional distribution within a block. One transition of $GPG$ consists of a $G$-move, a $P$-move, and a second $G$-move. If their costs are bounded by $c_G,c_P,c_G$, respectively, a direct comparison is between
\[
(c_P+2c_G)t_{F,\pi}(GPG,\varepsilon)
\quad\text{and}\quad
c_Pt_{F,\pi}(P,\varepsilon).
\]
When the reduction in mixing time compensates for the two conditional draws, using $GPG$ improves the sampling cost. The mixing-time inequality by itself compares transition counts.

\section{Tuning the \texorpdfstring{$k$}{k}-block partition to minimize
\texorpdfstring{$\Dphi(G\|\Pi)$}{the Bregman divergence from stationarity}}
\label{sec:partition-tuning}

Fix $k\in\llbracket 2,N\rrbracket$. In this section, the partition $\mathcal{O} = (\mathcal{O}_1,\ldots,\mathcal{O}_k)$ varies over all partitions of $\X$ into $k$ nonempty blocks, and $G = G_{\mathcal{O}}$ always denotes its induced Gibbs kernel. We study the objective $\Dphi(G\|\Pi)$, which depends only on the block masses. Setting $P=I$ in \eqref{eq:stationarity-Pythagoras} gives
\begin{equation}\label{eq:partition-tuning-Pythagoras}
\Dphi(I\|\Pi)=\Dphi(I\|G)+\Dphi(G\|\Pi).
\end{equation}
Thus minimizing $\Dphi(G_\mathcal{O}\|\Pi)$ over $\mathcal{O}$ is equivalent to maximizing $\Dphi(I\|G_{\mathcal{O}})$. For a fixed partition, this objective equals the largest value of $\Dphi(GPG\|\Pi)$ over all $P\in\mathcal S(\pi)$, as the next proposition shows.

\subsection{Worst-case interpretation and block-mass representation}

\begin{prop}[Worst-case divergence over stationary baselines]\label{prop:worst-case}
Let $\phi:[0,\infty)\to\R$ be convex and continuous, and continuously differentiable on $(0,\infty)$. For every partition with induced Gibbs kernel $G$ and every $P\in\mathcal S(\pi)$,
\begin{equation}\label{eq:worst-case}
\Dphi(GPG\|\Pi)\le\Dphi(G\|\Pi).
\end{equation}
Consequently,
\[
\max_{P\in\mathcal S(\pi)}\Dphi(GPG\|\Pi)=\Dphi(G\|\Pi),
\]
with the maximum attained at $P=I$, and
\[
\Dphi(GPG\|\Pi)
\le\min\bigl\{\Dphi(P\|\Pi),\,\Dphi(G\|\Pi)\bigr\}.
\]
If $\phi$ is strictly convex, equality in \eqref{eq:worst-case} holds if and only if $\overline P$ is a permutation matrix. Equivalently, there is a permutation $\sigma$ of $\llbracket k\rrbracket$ such that
\[
P(x,\mathcal O_{\sigma(i)})=1
\quad\text{for every }x\in\mathcal O_i,\qquad
\overline\pi(\sigma(i))=\overline\pi(i),
\]
where $P(x,\mathcal O_j):=\sum_{y\in\mathcal O_j}P(x,y)$.
\end{prop}

\begin{proof}
Each row of $GPG$ is a mixture of rows of $G$:
\[
(GPG)(x,y)=\sum_{z\in\X}(GP)(x,z)G(z,y).
\]
The function $B_\phi(\,\cdot\,,1)$ is convex. Jensen's inequality and the stationarity identity $\pi GP=\pi$ therefore give
\begin{align*}
\Dphi(GPG\|\Pi)
&\le\sum_{x,z\in\X}\pi(x)(GP)(x,z)
     \sum_{y\in\X}\pi(y)
     B_\phi\!\left(\frac{G(z,y)}{\pi(y)},1\right)\\
&=\sum_{z,y\in\X}\pi(z)\pi(y)
     B_\phi\!\left(\frac{G(z,y)}{\pi(y)},1\right)
=\Dphi(G\|\Pi).
\end{align*}
This proves \eqref{eq:worst-case}. Since $GIG=G$, the upper bound is attained at $P=I$. Combining it with Corollary~\ref{cor:stationarity-Pythagoras} proves the bound by the minimum of the two divergences.

Suppose now that $\phi$ is strictly convex. Since every $\pi(x)$ is positive, equality holds if and only if, for each $x$, all rows $G(z,\cdot)$ with $(GP)(x,z)>0$ coincide. Two rows of $G$ coincide exactly when their states belong to the same block. Thus $(GP)(x,\cdot)$ must be supported on one block. For $x\in\mathcal O_i$,
\[
(GP)(x,\cdot)
=\sum_{u\in\mathcal O_i}
\frac{\pi(u)}{\overline\pi(i)}P(u,\cdot),
\]
and every coefficient is positive. Hence equality is equivalent to the existence of a block index $\sigma(i)$ such that
$P(u,\mathcal O_{\sigma(i)})=1$ for all $u\in\mathcal O_i$. In other words, every row of $\overline P$ has a single entry equal to one. Stationarity gives
\[
\overline\pi(j)=\sum_{\{i:\sigma(i)=j\}}\overline\pi(i)>0,
\]
so every block has a preimage. A surjective map from $\llbracket k\rrbracket$ to itself is a permutation, and the same identity then gives
$\overline\pi(\sigma(i))=\overline\pi(i)$. Conversely, this deterministic block transition makes every row mixture involve identical rows of $G$, so equality holds.
\end{proof}

Proposition~\ref{prop:worst-case} gives a minimax interpretation of the partition objective:
\[
\min_{\substack{\mathcal O=(\mathcal O_1,\ldots,\mathcal O_k)\\\mathcal O_i\neq\emptyset,\ i\in\llbracket k\rrbracket}}\Dphi(G_{\mathcal O}\|\Pi)
=
\min_{\substack{\mathcal O=(\mathcal O_1,\ldots,\mathcal O_k)\\\mathcal O_i\neq\emptyset,\ i\in\llbracket k\rrbracket}}\max_{P\in\mathcal S(\pi)}
\Dphi(G_{\mathcal O}PG_{\mathcal O}\|\Pi).
\]
Thus choosing a partition to minimize $\Dphi(G_{\mathcal O}\|\Pi)$ minimizes a sharp upper bound on the divergence after averaging that holds uniformly over all stationary baselines.


By Proposition~\ref{prop:divergence-reduction}, inequality \eqref{eq:worst-case} can also be written on the block state space as
\[
\mathcal D_\phi^{\overline\pi}(\overline P\|\overline\Pi)
\le
\mathcal D_\phi^{\overline\pi}(I_k\|\overline\Pi),
\]
where $I_k$ is the identity transition matrix on $\llbracket k\rrbracket$. For the quadratic generator, this becomes
\[
\Fpi{GPG-\Pi}^2
=\norm{\overline P-\overline\Pi}_{F,\overline\pi}^2
\le k-1.
\]
In general, the middle quantity is the sum of the squared singular values of $\overline P$ on the centered space $\ell^2_0(\overline\pi)$. If $P$ is reversible and $1,\mu_2,\ldots,\mu_k$ are the eigenvalues of $\overline P$, it equals $\sum_{j=2}^k\mu_j^2$. In that case, equality holds exactly when $|\mu_j|=1$ for every $j\in\llbracket 2,k\rrbracket$.

We next express the objective in terms of the block masses.

\begin{prop}[Block-mass representation]\label{prop:G-block-mass}
Let $\phi:[0,\infty)\to\R$ be convex and continuous, and continuously differentiable on $(0,\infty)$. For a partition $\mathcal O=(\mathcal O_1,\ldots,\mathcal O_k)$, recall that $\overline{\pi}(i)=\pi(\mathcal O_i)$ and define
\begin{equation}\label{eq:mass-profile}
h_\phi(a):=a^2\bigl[\phi(1/a)-\phi(0)\bigr]
+a\bigl[\phi(0)-\phi(1)\bigr],
\qquad 0<a\le1.
\end{equation}
Then
\begin{align}
\Dphi(G\|\Pi)
&=\sum_{i=1}^k \overline{\pi}(i)^2 B_\phi(1/\overline{\pi}(i),1)
+\left(1-\sum_{i=1}^k \overline{\pi}(i)^2\right)B_\phi(0,1)
\label{eq:G-block-Bregman}\\
&=\sum_{i=1}^k \overline{\pi}(i)^2\bigl[\phi(1/\overline{\pi}(i))-\phi(0)\bigr]
+\phi(0)-\phi(1)
=\sum_{i=1}^k h_\phi(\overline{\pi}(i)).
\label{eq:G-block-mass}
\end{align}
No affine normalization of $\phi$ is required.
\end{prop}

\begin{proof}
The transition density $G(x,y)/\pi(y)$ equals $1/\overline{\pi}(i)$ on $\mathcal O_i\times\mathcal O_i$ and equals zero when $x$ and $y$ belong to different blocks. The $\pi\otimes\pi$ mass of $\mathcal O_i\times\mathcal O_i$ is $\overline{\pi}(i)^2$, which gives \eqref{eq:G-block-Bregman}. On expanding the scalar Bregman divergences, the coefficient of $-\phi'(1)$ is
\[
\sum_{i=1}^k \overline{\pi}(i)^2(1/\overline{\pi}(i)-1)
-\left(1-\sum_{i=1}^k \overline{\pi}(i)^2\right)
=\sum_{i=1}^k \overline{\pi}(i)-1=0.
\]
Collecting the remaining terms proves \eqref{eq:G-block-mass}.
\end{proof}

Adding an affine function to $\phi$ leaves $h_\phi$ unchanged. In particular, the formula applies directly to unnormalized generators such as $u\log u$ and $u^p$. Applying it also to the singleton partition together with \eqref{eq:partition-tuning-Pythagoras} gives
\begin{equation}\label{eq:I-G-block-mass}
\Dphi(I\|G)
=\sum_{x\in\X}h_\phi(\pi(x))-\sum_{i=1}^k h_\phi(\overline{\pi}(i)).
\end{equation}

\subsection{Concave block costs: an exact algorithm}

Suppose that $h_\phi$ is concave on $(0,1]$. The following algorithm produces an optimal partition.
\begin{algorithm}[H]
\caption{Exact partition for concave block costs}
\label{alg:concave-partition}
\textbf{Input:} State masses $\pi(x)$, $x\in\X$, and $k\in\llbracket 2,N\rrbracket$.

\textbf{Output:} A partition $\mathcal O=(\mathcal O_1,\ldots,\mathcal O_k)$ minimizing $\Dphi(G_{\mathcal O}\|\Pi)$ when $h_\phi$ is concave.

\begin{enumerate}[label=\arabic*.,leftmargin=*]
\item Order the states as $x_1,\ldots,x_N$ so that
$\pi(x_1)\le\cdots\le\pi(x_N)$.
\item Make each of $x_1,\ldots,x_{k-1}$ a singleton block, and put all remaining states in the last block:
\begin{equation}\label{eq:concave-optimal-partition}
\mathcal O_i:=\{x_i\},\quad i\in\llbracket k-1\rrbracket,
\qquad
\mathcal O_k:=\{x_k,\ldots,x_N\}.
\end{equation}
\end{enumerate}
\end{algorithm}
Ties in Algorithm~\ref{alg:concave-partition} may be resolved arbitrarily. Sorting takes $O(N\log N)$ comparisons, and the remaining work is linear in $N$.

\begin{theorem}[Exact minimization for concave block costs]\label{thm:concave-partition}
Under the assumptions of Proposition~\ref{prop:G-block-mass}, suppose that $h_\phi$ is concave on $(0,1]$. The partition \eqref{eq:concave-optimal-partition} minimizes $\Dphi(G_{\mathcal O}\|\Pi)$ among all $k$-block partitions. Its value is
\begin{equation*}
\min_{\substack{\mathcal O=(\mathcal O_1,\ldots,\mathcal O_k)\\\mathcal O_i\neq\emptyset,\ i\in\llbracket k\rrbracket}}\Dphi(G_{\mathcal O}\|\Pi)
=\sum_{i=1}^{k-1}h_\phi(\pi(x_i))
+h_\phi\!\left(1-\sum_{i=1}^{k-1}\pi(x_i)\right).
\end{equation*}
\end{theorem}

\begin{proof}
Concavity implies the transfer inequality
\begin{equation}\label{eq:concave-mass-transfer}
h_\phi(a+v)+h_\phi(b-v)
\le h_\phi(a)+h_\phi(b),
\qquad a\ge b>0,\quad 0\le v<b,\quad a+v\le1.
\end{equation}
Indeed, for $v>0$, the secant slope on $[b-v,b]$ is at least the secant slope on $[a,a+v]$. Multiplication by $v$ gives the inequality; the case $v=0$ is immediate.

Start from any $k$-block partition and designate a block of largest mass. In each of the other blocks, retain a state of smallest mass and transfer all its other states to the designated block. At each transfer, \eqref{eq:concave-mass-transfer} applies because the designated block remains largest. The objective cannot increase, and all blocks remain nonempty.

There are now $k-1$ singleton blocks and one designated block. If a singleton state $x$ satisfies $\pi(x)>\pi(y)$ for some state $y$ in the designated block, interchange these two states. If the designated block has mass $a$, the affected masses change from $(a,\pi(x))$ to $(a+\pi(x)-\pi(y),\pi(y))$. Since $a\ge\pi(x)$, another application of \eqref{eq:concave-mass-transfer} shows that the objective cannot increase. There are finitely many choices of singleton states, and each interchange strictly decreases their total mass. Hence only finitely many interchanges are possible. At termination, the singleton masses are the $k-1$ smallest state masses, up to ties. Thus the output in \eqref{eq:concave-optimal-partition} has objective no greater than that of any initial partition, proving the claim.
\end{proof}

If $h_\phi$ has a concave extension to $[0,1]$, then $S\mapsto h_\phi(\pi(S))$ is a submodular set function. This follows from the composition rule for a concave function and a modular function with nonnegative weights; see \citet[Proposition~6.1]{Bach2013}. Theorem~\ref{thm:concave-partition} therefore solves this particular submodular partition problem exactly. Its proof uses the fact that every block has the same cost function of its mass and that there are no prescribed terminals or additional block constraints.

Table~\ref{tab:concave-generators} lists generators for which the theorem applies. In every row, $h_\phi$ extends continuously to zero with $h_\phi(0):=0$ and is concave. These assertions follow by differentiating the displayed profiles. The expression in the first row is the Shannon entropy of the block distribution.

\begin{table}[htbp]
\centering
\caption{Concave block costs for admissible convex generators. Each objective is minimized by isolating the $k-1$ states of smallest $\pi$-mass. We use $0\log0:=0$.}
\label{tab:concave-generators}
\small
\renewcommand{\arraystretch}{1.55}
\begin{tabular}{@{}lll@{}}
\toprule
Generator $\phi(u)$ & Block cost $h_\phi(a)$ & $\Dphi(G\|\Pi)$\\
\midrule
$u\log u$ & $-a\log a$ & $-\displaystyle\sum_{i=1}^k \overline{\pi}(i)\log \overline{\pi}(i)=H(\overline\pi)$\\
$u^p$, $1<p<2$ & $a^{2-p}-a$ & $\displaystyle\sum_{i=1}^k \overline{\pi}(i)^{2-p}-1$\\
$-u^p$, $0<p<1$ & $a-a^{2-p}$ & $1-\displaystyle\sum_{i=1}^k \overline{\pi}(i)^{2-p}$\\
$(\sqrt u-1)^2$ & $2(a-a^{3/2})$ & $2\left(1-\displaystyle\sum_{i=1}^k \overline{\pi}(i)^{3/2}\right)$\\
\bottomrule
\end{tabular}
\end{table}

\FloatBarrier
\begin{remark*}[Tsallis entropy]
For a finite probability distribution $\nu$ with full support, its Tsallis entropy of order $q\in\R\setminus\{1\}$ is
\[
H_q(\nu):=\frac{1-\sum_x\nu(x)^q}{q-1};
\]
see \citet{Tsallis1988,Tsallis2000}. We set $H_1(\nu):=H(\nu)$ by continuity. For $\phi(u):=u^p$ with $1<p<2$, Proposition~\ref{prop:G-block-mass} gives
\[
\mathcal D_{u^p}^{\pi}(G\|\Pi)
=\sum_{i=1}^k\overline{\pi}(i)^{2-p}-1
=(p-1)H_{2-p}(\overline{\pi}).
\]
Thus minimizing this power divergence is equivalent to minimizing the Tsallis entropy of the block distribution, with order $2-p\in(0,1)$. Algorithm~\ref{alg:concave-partition} therefore minimizes this entropy over all $k$-block partitions simultaneously for every order in $(0,1)$. The generator $\phi(u):=-u^p$ in Table~\ref{tab:concave-generators} has the analogous interpretation:
$\mathcal D_{-u^p}^{\pi}(G\|\Pi)=(1-p)H_{2-p}(\overline{\pi})$ for $0<p<1$.

For each fixed partition, the Shannon entropy objective is recovered by the normalized limit
\[
\lim_{p\downarrow1}\frac{\mathcal D_{u^p}^{\pi}(G\|\Pi)}{p-1}
=H(\overline{\pi})
=\mathcal D_{u\log u}^{\pi}(G\|\Pi).
\]
Indeed, $H_q(\nu)\to H(\nu)$ as $q\to1$. The normalization by $p-1$ is essential, since the unnormalized power divergence tends to zero as $p\downarrow1$.
\end{remark*}

\subsection{Powers above two: a greedy approximation}

Let $\phi(u):=u^p$ with $p>2$. Proposition~\ref{prop:G-block-mass} gives
\begin{equation}\label{eq:inverse-power-objective}
\mathcal D_{u^p}^{\pi}(G\|\Pi)
=\sum_{i=1}^k\overline{\pi}(i)^{2-p}-1.
\end{equation}
The function $(0,1] \ni a\mapsto a^{2-p}$ is strictly convex. Jensen's inequality consequently gives
\begin{equation}\label{eq:power-balanced-bound}
\sum_{i=1}^k\overline{\pi}(i)^{2-p}\ge k^{p-1},
\qquad
\mathcal D_{u^p}^{\pi}(G\|\Pi)\ge k^{p-1}-1,
\end{equation}
with equality if and only if all block masses equal $1/k$. Equal block masses may not be attainable because the state masses cannot be split.

The sum in \eqref{eq:inverse-power-objective} is a supermodular partition cost. To make this statement precise even at the empty set, define
\begin{equation}
\psi_p(t):=
\begin{cases}
1+(p-2)(1-t),&0\le t\le1,\\
t^{2-p},&t\ge1.
\end{cases}
\end{equation}
This function is convex and continuously differentiable. Hence
\[
f_p(S):=\left(\min_{x\in\X}\pi(x)\right)^{2-p}
\psi_p\!\left(\frac{\pi(S)}{\min_{x\in\X}\pi(x)}\right),
\qquad S\subseteq\X,
\]
is supermodular. Indeed, $-\psi_p$ is concave, and $S\mapsto\pi(S)/\min_{x\in\X}\pi(x)$ is modular with nonnegative weights. Applying \citet[Proposition~6.1]{Bach2013} and multiplying by the positive factor $(\min_{x\in\X}\pi(x))^{2-p}$ shows that $-f_p$ is submodular, which is equivalent to supermodularity of $f_p$. For every nonempty $S$, $f_p(S)=\pi(S)^{2-p}$, while $f_p(\varnothing)=(p-1)\left(\min_{x\in\X}\pi(x)\right)^{2-p}$ is finite. Thus minimizing \eqref{eq:inverse-power-objective} is equivalent to minimizing $\sum_i f_p(\mathcal O_i)$ over nonempty $k$-block partitions.

A simple algorithm assigns the largest state masses first to the currently lightest block:
\begin{algorithm}[H]
\caption{Greedy partition for inverse-power costs}
\label{alg:inverse-power-greedy}
\textbf{Input:} State masses $\pi(x)$, $x\in\X$, and $k\in\llbracket 2,N\rrbracket$.

\textbf{Output:} A partition $\mathcal O=(\mathcal O_1,\ldots,\mathcal O_k)$ into $k$ nonempty blocks.

\begin{enumerate}[label=\arabic*.,leftmargin=*]
\item Order the states so that $\pi(x_1)\ge\cdots\ge\pi(x_N)$.
\item Initialize $\mathcal O_i:=\{x_i\}$ for $i\in\llbracket k\rrbracket$.
\item For $j\in\llbracket k+1,N\rrbracket$ in increasing order, put $x_j$ in a block with smallest current $\pi$-mass.
\end{enumerate}
\end{algorithm}
Ties may be resolved arbitrarily. Algorithm~\ref{alg:inverse-power-greedy} is the longest-processing-time rule from scheduling, with state masses as job sizes and blocks as machines; see \citet{AlonEtAl1998}. Sorting and maintaining a priority queue of block masses take $O(N\log N+N\log k)$ comparisons and arithmetic operations. Write $(\mathcal O_i^{\mathrm{gr}})_{i\in\llbracket k\rrbracket}$ for the output partition and $G_{\mathrm{gr}}$ for its Gibbs kernel.

Algorithm~\ref{alg:inverse-power-greedy} does not depend on $p$: it uses only the state masses, the number of blocks, and the tie-breaking rule. For a fixed tie-breaking rule, the same output partition therefore satisfies Theorem~\ref{thm:inverse-power-greedy} simultaneously for every $p>2$, with the approximation factor $\alpha_p$ varying with $p$.

\begin{theorem}[Greedy approximation for inverse-power costs]\label{thm:inverse-power-greedy}
For $p>2$, define
\begin{equation*}
\alpha_p:=
\frac{(p-2)^{p-2}(2-2^{2-p})^{p-1}}
{(p-1)^{p-1}(1-2^{2-p})^{p-2}}.
\end{equation*}
The greedy partition satisfies
\begin{equation}\label{eq:inverse-power-greedy}
\sum_{i=1}^k\pi(\mathcal O_i^{\mathrm{gr}})^{2-p}
\le\alpha_p\min_{\substack{\mathcal O=(\mathcal O_1,\ldots,\mathcal O_k)\\\mathcal O_i\neq\emptyset,\ i\in\llbracket k\rrbracket}}
\sum_{i=1}^k\overline{\pi}(i)^{2-p}.
\end{equation}
Consequently,
\begin{align}
\mathcal D_{u^p}^{\pi}(G_{\mathrm{gr}}\|\Pi)
&\le\alpha_p \min_{\substack{\mathcal O=(\mathcal O_1,\ldots,\mathcal O_k)\\\mathcal O_i\neq\emptyset,\ i\in\llbracket k\rrbracket}}\mathcal D_{u^p}^{\pi}(G_{\mathcal O}\|\Pi)+\alpha_p-1
\label{eq:power-greedy-additive}\\
&\le\left(\alpha_p+\frac{\alpha_p-1}{k^{p-1}-1}\right)\min_{\substack{\mathcal O=(\mathcal O_1,\ldots,\mathcal O_k)\\\mathcal O_i\neq\emptyset,\ i\in\llbracket k\rrbracket}}\mathcal D_{u^p}^{\pi}(G_{\mathcal O}\|\Pi).
\label{eq:power-greedy-ratio}
\end{align}
In particular, $\alpha_3=9/8$ and $\alpha_4=343/243$.
\end{theorem}

\begin{proof}
We first remove states that must be singletons. Consider any residual instance with total mass $W$ and $\ell\ge2$ blocks, and put $\mu:=W/\ell$. If a state $x$ satisfies $\pi(x)\ge\mu$, every minimizing partition makes it a singleton. Indeed, suppose its block has mass $\pi(x)+v$ with $v>0$. Some other block has mass $b<\pi(x)$. Moving the other states of total mass $v$ into that block changes the two masses to $\pi(x)$ and $b+v$, and strictly lowers the cost because
\[
b^{2-p}-(b+v)^{2-p}
>
\pi(x)^{2-p}-(\pi(x)+v)^{2-p}.
\]
The greedy algorithm also leaves this state alone: before it could assign another state to that block, all $\ell$ current block masses would be at least $\pi(x)$, so the mass already assigned would be at least $\ell\pi(x)\ge W$, although a positive mass would remain unassigned. This is impossible.

Remove such singleton states and their blocks successively. Their costs are the same for the greedy and optimal partitions, and restricting the greedy assignments to the remaining states gives the greedy algorithm for the residual instance. If one block remains, it is assigned exactly. Otherwise, we reach an instance in which every state mass is smaller than its average block mass $\mu$.

Let $m$ be the smallest final greedy block mass in this residual instance. Every nonsingleton greedy block has mass at most $2m$. To prove this, consider the last state assigned to that block. Just before this assignment, its block was lightest. Its mass was therefore at most the then-current mass of a block whose final mass is $m$, and hence at most $m$. Because the states were processed in decreasing order, the mass of this last state is at most that of a state already in the block, and hence is also at most $m$. Every singleton block has mass smaller than $\mu$. If $\mu>2m$, all final block masses would be smaller than $\mu$, contradicting their average. Therefore $\mu\le2m$, and every final greedy block mass lies in $[m,2m]$.

For $m\le a\le2m$, the chord bound for the convex function $a^{2-p}$ gives
\[
a^{2-p}\le m^{2-p}
\left[2-\frac am+\left(\frac am-1\right)2^{2-p}\right].
\]
Write $z:=\mu/m\in[1,2]$. Summing this bound over the $\ell$ residual blocks and using Jensen's lower bound $\ell\mu^{2-p}$ for their optimal cost shows that their approximation ratio is at most
\[
z^{p-2}\bigl[(2-z)+(z-1)2^{2-p}\bigr].
\]
Differentiation shows that its maximum on $[1,2]$ occurs at
\[
z=\frac{(p-2)(2-2^{2-p})}{(p-1)(1-2^{2-p})}\in(1,2)
\]
and equals $\alpha_p$. Since $\alpha_p\ge1$, restoring the common singleton costs preserves the same approximation factor. This proves \eqref{eq:inverse-power-greedy}.

Finally, subtracting one from the objective gives \eqref{eq:power-greedy-additive}. The lower bound
\[
\min_{\substack{\mathcal O=(\mathcal O_1,\ldots,\mathcal O_k)\\\mathcal O_i\neq\emptyset,\ i\in\llbracket k\rrbracket}}
\mathcal D_{u^p}^{\pi}(G_{\mathcal O}\|\Pi)\ge k^{p-1}-1
\]
from \eqref{eq:power-balanced-bound} then proves \eqref{eq:power-greedy-ratio}.
\end{proof}

For example, when $p=3$, the guarantee becomes
\begin{align*}
\mathcal D_{u^3}^{\pi}(G_{\mathrm{gr}}\|\Pi)
&\le\left(\frac98+\frac{1}{8(k^2-1)}\right)
\min_{\substack{\mathcal O=(\mathcal O_1,\ldots,\mathcal O_k)\\\mathcal O_i\neq\emptyset,\ i\in\llbracket k\rrbracket}}\mathcal D_{u^3}^{\pi}(G_{\mathcal O}\|\Pi)\\
&\le\frac76\min_{\substack{\mathcal O=(\mathcal O_1,\ldots,\mathcal O_k)\\\mathcal O_i\neq\emptyset,\ i\in\llbracket k\rrbracket}}\mathcal D_{u^3}^{\pi}(G_{\mathcal O}\|\Pi).
\end{align*}
\begin{remark*}[Tsallis entropy of negative order]
The entropy interpretation also applies when $p>2$. Since $\pi$ has full support and all blocks are nonempty, every $\overline{\pi}(i)$ is strictly positive. Hence the Tsallis entropy of order $2-p<0$ is finite, and \eqref{eq:inverse-power-objective} gives
\[
\mathcal D_{u^p}^{\pi}(G\|\Pi)
=(p-1)H_{2-p}(\overline{\pi}).
\]
Thus minimizing the power divergence is equivalent to minimizing Tsallis entropy of negative order. For $q<0$, the scalar cost $a\mapsto(a^q-a)/(1-q)$ has second derivative $-q a^{q-2}>0$. Consequently, $H_q$ is strictly convex on the positive probability simplex and is minimized by the uniform distribution; see also \citet[Section~II.A]{Tsallis2000}. In particular,
\[
H_{2-p}(\overline{\pi})
\ge\frac{k^{p-1}-1}{p-1},
\]
with equality if and only if every block has mass $1/k$. Moreover, this entropy tends to infinity as any block mass tends to zero. These properties explain why the objective favors balanced masses for $p>2$.

The multiplicative guarantee in \eqref{eq:power-greedy-ratio} also holds with $\mathcal D_{u^p}^{\pi}(G\|\Pi)$ replaced by $H_{2-p}(\overline{\pi})$, since the positive factor $p-1$ cancels. Thus Algorithm~\ref{alg:inverse-power-greedy} also approximates the minimum Tsallis entropy of negative order over nonempty $k$-block partitions.
\end{remark*}

\subsection{The quadratic case and the balance--imbalance transition}

\begin{prop}[Partition independence for the quadratic generator]\label{prop:quadratic-G-rank}
For $\phi(u):=u^2$, every nonempty $k$-block partition satisfies
\begin{equation*}
\mathcal D_{u^2}^{\pi}(G\|\Pi)=\Fpi{G-\Pi}^2=k-1,
\qquad
\mathcal D_{u^2}^{\pi}(I\|G)=N-k.
\end{equation*}
In particular, all $k$-block partitions minimize the first divergence.
\end{prop}

\begin{proof}
The profile in \eqref{eq:mass-profile} is $h_\phi(a)=1-a$, so \eqref{eq:G-block-mass} gives $k-\sum_i \overline{\pi}(i)=k-1$. Equation~\eqref{eq:I-G-block-mass} gives $(N-1)-(k-1)=N-k$. Equivalently, $G-\Pi$ and $I-G$ are orthogonal projections of ranks $k-1$ and $N-k$.
\end{proof}

For the power family $\phi(u) = u^p$, $p>1$, the entire transition follows from
\begin{equation*}
h_\phi(a)=a^{2-p}-a,
\qquad
h_\phi''(a)=(p-1)(p-2)a^{-p}.
\end{equation*}
When $1<p<2$, transferring mass from a smaller block to a larger block decreases the objective, provided neither block becomes empty. For $p>2$, the same transfer increases it. At $p=2$, the objective is unchanged. In terms of Tsallis entropy, these regimes correspond to positive, negative, and zero orders, respectively: the identity $\mathcal D_{u^p}^{\pi}(G\|\Pi)=(p-1)H_{2-p}(\overline{\pi})$ holds for every $p>1$, including $p=2$, where $H_0(\overline{\pi})=k-1$. Table~\ref{tab:power-transition} summarizes these regimes.

\begin{table}[htbp]
\centering
\caption{Balance--imbalance transition for convex power generators $\phi(u):=u^p$, $p>1$. Here $q:=2-p$ and $\mathcal D_{u^p}^{\pi}(G\|\Pi)=(p-1)H_q(\overline{\pi})$. All block masses are positive, so the entropy is finite even for $q<0$. The lower bound in the last row is $k^{p-1}-1$, attained exactly when every block has mass $1/k$.}
\label{tab:power-transition}
\small
\renewcommand{\arraystretch}{1.4}
\begin{tabular}{@{}p{0.15\textwidth}p{0.27\textwidth}p{0.50\textwidth}@{}}
\toprule
Power $p$\newline Order $q$ & $\mathcal D_{u^p}^{\pi}(G\|\Pi)$ & Minimization over nonempty $k$-block partitions\\
\midrule
$1<p<2$\newline $0<q<1$ &
$\displaystyle\sum_{i=1}^k \overline{\pi}(i)^{2-p}-1$ &
Minimizes concave Tsallis entropy of positive order; favors imbalance. Isolating the $k-1$ lightest states is exactly optimal.\\
$p=2$\newline $q=0$ &
$k-1=H_0(\overline{\pi})$ &
The order-zero entropy is independent of the positive block masses. Every partition is optimal.\\
$p>2$\newline $q<0$ &
$\displaystyle\sum_{i=1}^k \overline{\pi}(i)^{-(p-2)}-1$ &
Minimizes strictly convex Tsallis entropy of negative order; favors balance. Equal masses are optimal when feasible. Otherwise, the greedy guarantee applies.\\
\bottomrule
\end{tabular}
\end{table}

For $p>2$, favoring balance does not specify a single optimal discrete partition for every $p$: feasible mass vectors need not be ordered by transfers of mass from smaller to larger blocks. For $p=1$, the generator is affine and the divergence is identically zero. The raw power $u^p$ is not a convex generator when $0<p<1$; the convex choice $-u^p$ is included in Table~\ref{tab:concave-generators}.

\section{Numerical experiment}\label{sec:numerical}

The script used
for the experiments and all figures is available at
\url{https://github.com/mchchoi/bregman-projections}.

We compare a lazy Metropolis sampler with two Gibbs-projected versions on a cycle with $N=100$ states. Let $\X:=\llbracket 0,99\rrbracket$, with addition modulo $100$, and set
\[
H(x):=-\cos\!\left(\frac{8\pi x}{100}\right)
       -0.15\cos\!\left(\frac{2\pi x}{100}\right),
\qquad
\pi(x):=\frac{e^{-2H(x)}}{\sum_{z\in\X}e^{-2H(z)}}.
\]
This Gibbs distribution has four wells at inverse temperature $2$. The baseline kernel proposes a neighboring state and accepts according to the Metropolis rule:
\[
P(x,x\pm1):=\frac14\min\!\left\{1,\frac{\pi(x\pm1)}{\pi(x)}\right\},
\qquad
P(x,x):=1-P(x,x-1)-P(x,x+1),
\]
with all other entries zero. Thus $P$ is reversible and ergodic, and its holding probability is at least $1/2$.

For $k=4$, let $G_1$ be induced by Algorithm~\ref{alg:concave-partition} for $\phi(u):=u^{1.5}$, and let $G_2$ be induced by Algorithm~\ref{alg:inverse-power-greedy} for $\phi(u):=u^3$. Ties are resolved by increasing state index, and then by increasing block index. The partition for $G_1$ consists of the singleton blocks $\{38\}$, $\{62\}$, $\{37\}$, and their complement. Its largest block has $97$ states and mass approximately $0.9985719$. In contrast, the greedy partition has $25$ states in each block, with masses
\[
(0.25000802,\ 0.25000520,\ 0.24999398,\ 0.24999280).
\]
The corresponding objective values are
\[
\mathcal D_{u^{1.5}}^\pi(G_1\|\Pi)\approx0.06473985,
\qquad
\mathcal D_{u^3}^\pi(G_2\|\Pi)\approx15.0000000115.
\]
The first value is optimal by Theorem~\ref{thm:concave-partition}; the second is within $1.2\times10^{-8}$ of the lower bound $4^2-1=15$ in \eqref{eq:power-balanced-bound}.

\begin{figure}[htbp]
\centering
\includegraphics[width=\textwidth]{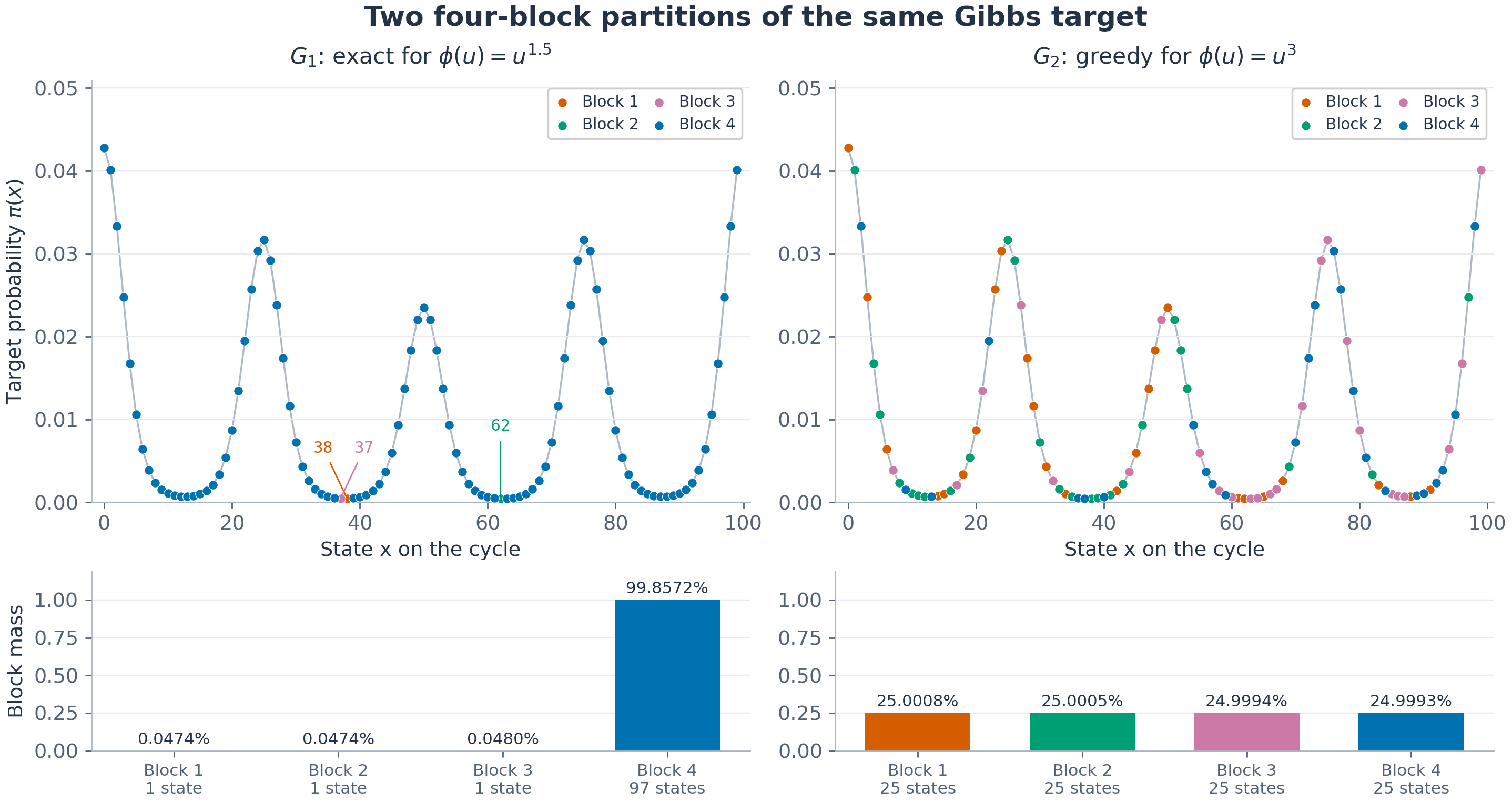}
\caption{The two partitions of the same target distribution. In the upper panels, colors identify block membership; the three singleton states for $G_1$ are annotated. The lower panels show block masses and numbers of states. Block colors are local to each partition.}
\label{fig:numerical-partitions}
\end{figure}

Figure~\ref{fig:numerical-partitions} displays the balance--imbalance transition from Section~\ref{sec:partition-tuning}: the concave cost isolates rare states, whereas the inverse-power cost favors equal block masses.

For each $Q\in\{P,G_1PG_1,G_2PG_2\}$, we compute the worst-case total variation distance and the squared-Frobenius distance,
\[
\frac12\max_{x\in\X}\sum_{y\in\X}|Q^t(x,y)-\pi(y)|,
\qquad
\Fpi{Q^t-\Pi}^2.
\]
We use reversible spectral decompositions in double precision, checked against direct powers and the four-state projection chains. No trajectory sampling is involved. At time zero, all three powers equal $I$ and their squared-Frobenius distance is $99$.

\begin{figure}[p]
\centering
\includegraphics[width=0.78\textwidth]{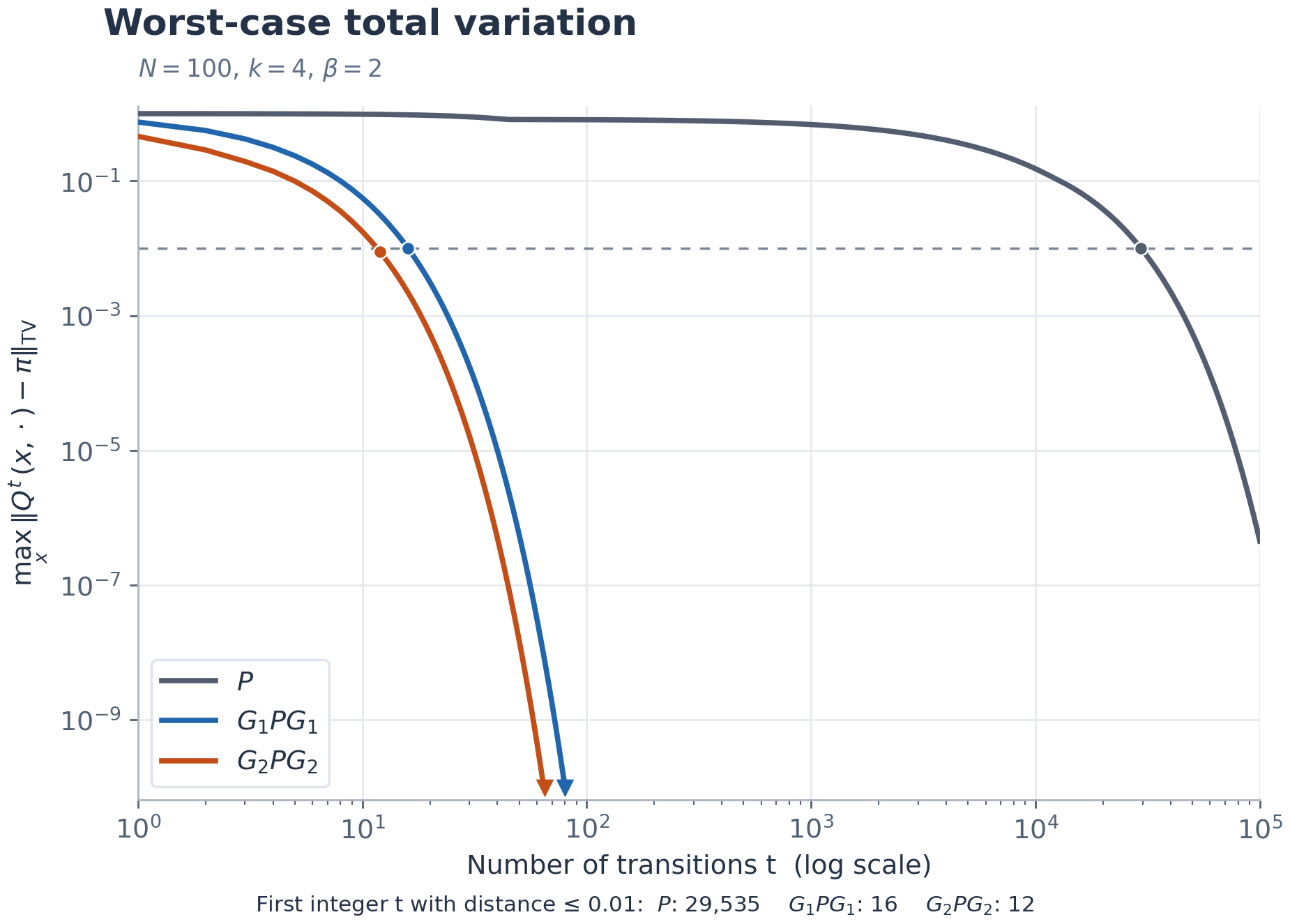}
\par\medskip
\includegraphics[width=0.78\textwidth]{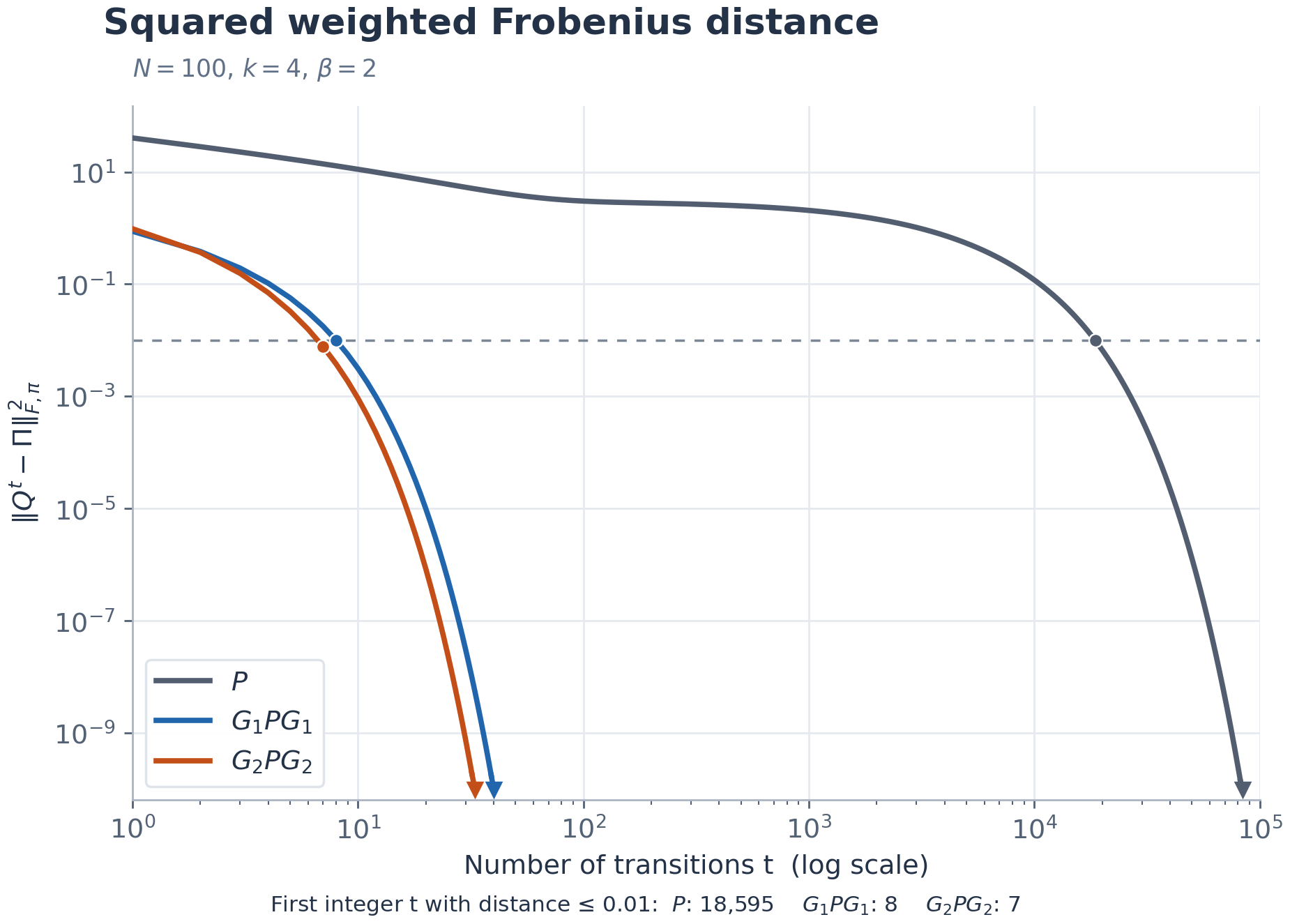}
\caption{Convergence to stationarity. Top: worst-case total variation distance, maximized over all $100$ initial states. Bottom: squared-Frobenius distance with the $\pi$-weights of Section~\ref{sec:preliminaries}. Both axes are logarithmic. Dashed lines mark $10^{-2}$, and circles mark the first integer times at or below this threshold. Downward triangles indicate that a curve has passed below the display limit $10^{-10}$.}
\label{fig:numerical-convergence}
\end{figure}

Figure~\ref{fig:numerical-convergence} shows faster convergence for both projected kernels. Table~\ref{tab:numerical-mixing} reports threshold times, each checked at the crossing and the preceding integer.

\begin{table}[htbp]
\centering
\caption{Mixing times at tolerance $10^{-2}$, measured in transitions of the indicated kernel. The squared-Frobenius column reports $t_{F,\pi}(Q,10^{-2})$.}
\label{tab:numerical-mixing}
\begin{tabular}{@{}lrr@{}}
\toprule
Kernel & Worst-case TV & Squared-Frobenius\\
\midrule
$P$ & $29{,}535$ & $18{,}595$\\
$G_1PG_1$ & $16$ & $8$\\
$G_2PG_2$ & $12$ & $7$\\
\bottomrule
\end{tabular}
\end{table}

Although $G_2PG_2$ reaches the tolerance earlier in Table~\ref{tab:numerical-mixing}, the squared-Frobenius curves cross between times $1$ and $2$: \(G_1PG_1\) has the smaller squared-Frobenius distance after one step; this ordering reverses after two steps. The algorithms optimize $\Dphi(G\|\Pi)$, not mixing time. Theorem~\ref{thm:all-time} guarantees the squared-Frobenius improvement over $P$; the TV comparison is a numerical observation.

These are transition counts, not runtimes: each projected step requires two additional conditional draws, and partition construction enumerates all $100$ states. 


\FloatBarrier

\backmatter

\section*{Declarations}

\bmhead{Funding}
Michael Choi acknowledges the financial support of the projects A-0000178-02-00 and A-8003574-00-00 at National University of Singapore.

\bmhead{Competing interests}
Both authors have no relevant financial or non-financial interests to disclose.

\bmhead{Data availability}
No data was used for the research described in the article.

\bmhead{Author contributions}
Michael Choi and Youjia Wang jointly contributed to idea formulation, execution, and manuscript writing. Michael Choi supervised the project. 

\bmhead{Declaration of AI usage}
The initial research ideas, scientific methodology and analysis, drafting of the manuscript, and coding of the numerical experiments were carried out by the authors. Artificial intelligence (AI) tools were used solely to assist with language editing, review of the manuscript, and review of the numerical code. The authors reviewed and verified all AI-assisted outputs and made all final decisions regarding the content of the manuscript. All scientific content, interpretations, and conclusions remain the sole responsibility of the authors.



\begin{thebibliography}{99}

\bibitem[Adam\v{c}\'ik(2014)]{Adamcik2014}
M.~Adam\v{c}\'ik.
\newblock The information geometry of Bregman divergences and some applications in multi-expert reasoning.
\newblock \emph{Entropy}, 16(12):6338--6381, 2014.
\newblock \href{https://doi.org/10.3390/e16126338}{doi:10.3390/e16126338}.

\bibitem[Alon et~al.(1998)Alon, Azar, Woeginger, and Yadid]{AlonEtAl1998}
N.~Alon, Y.~Azar, G.~J. Woeginger, and T.~Yadid.
\newblock Approximation schemes for scheduling on parallel machines.
\newblock \emph{Journal of Scheduling}, 1(1):55--66, 1998.

\bibitem[Bach(2013)]{Bach2013}
F.~Bach.
\newblock Learning with submodular functions: a convex optimization perspective.
\newblock \emph{Foundations and Trends in Machine Learning}, 6(2--3):145--373, 2013.
\newblock \href{https://doi.org/10.1561/2200000039}{doi:10.1561/2200000039}.

\bibitem[Banerjee et~al.(2007)Banerjee, Dhillon, Ghosh, Merugu, and Modha]{BanerjeeEtAl2007}
A.~Banerjee, I.~S. Dhillon, J.~Ghosh, S.~Merugu, and D.~S. Modha.
\newblock A generalized maximum entropy approach to Bregman co-clustering and matrix approximation.
\newblock \emph{Journal of Machine Learning Research}, 8:1919--1986, 2007.
\newblock \url{https://jmlr.org/papers/v8/banerjee07a.html}.

\bibitem[Banerjee et~al.(2005)Banerjee, Guo, and Wang]{BanerjeeGuoWang2005}
A.~Banerjee, X.~Guo, and H.~Wang.
\newblock On the optimality of conditional expectation as a Bregman predictor.
\newblock \emph{IEEE Transactions on Information Theory}, 51(7):2664--2669, 2005.
\newblock \href{https://doi.org/10.1109/TIT.2005.850145}{doi:10.1109/TIT.2005.850145}.

\bibitem[Bregman(1967)]{Bregman1967}
L.~M. Bregman.
\newblock The relaxation method of finding the common point of convex sets and its application to the solution of problems in convex programming.
\newblock \emph{USSR Computational Mathematics and Mathematical Physics}, 7(3):200--217, 1967.
\newblock \href{https://doi.org/10.1016/0041-5553(67)90040-7}{doi:10.1016/0041-5553(67)90040-7}.

\bibitem[Choi and Wang(2025)]{ChoiWang2025}
M.~C.~H. Choi and Y.~Wang.
\newblock Group-averaged Markov chains: mixing improvement.
\newblock arXiv:2509.02996, 2025.
\newblock \url{https://arxiv.org/abs/2509.02996}.

\bibitem[Choi and Wang(2026)]{ChoiWangSpectral2026}
M.~C.~H. Choi and Y.~Wang.
\newblock Spectral partitioning for $k$-block averaging kernels of finite Markov chains.
\newblock arXiv:2608.21466, 2026.
\newblock \href{https://doi.org/10.48550/arXiv.2608.21466}{doi:10.48550/arXiv.2608.21466}.

\bibitem[Choi et~al.(2025)Choi, Lim, and Wang]{ChoiLimWang2025}
M.~C.~H. Choi, R.~J.~Y. Lim, and Y.~Wang.
\newblock Group-averaged Markov chains II: tuning of group action in finite state space.
\newblock arXiv:2512.13067, 2025. Revised September 2026.
\newblock \url{https://arxiv.org/abs/2512.13067}.

\bibitem[Dhillon and Tropp(2007)]{DhillonTropp2007}
I.~S. Dhillon and J.~A. Tropp.
\newblock Matrix nearness problems with Bregman divergences.
\newblock \emph{SIAM Journal on Matrix Analysis and Applications}, 29(4):1120--1146, 2007.
\newblock \href{https://doi.org/10.1137/060649021}{doi:10.1137/060649021}.

\bibitem[Golub and Van Loan(2013)]{GolubVanLoan2013}
G.~H. Golub and C.~F. Van Loan.
\newblock \emph{Matrix Computations}.
\newblock Johns Hopkins University Press, fourth edition, 2013.
\newblock \href{https://doi.org/10.1137/1.9781421407944}{doi:10.1137/1.9781421407944}.

\bibitem[Jerrum et~al.(2004)Jerrum, Son, Tetali, and Vigoda]{JerrumEtAl2004}
M.~Jerrum, J.-B.~Son, P.~Tetali, and E.~Vigoda.
\newblock Elementary bounds on Poincar\'e and log-Sobolev constants for decomposable Markov chains.
\newblock \emph{The Annals of Applied Probability}, 14(4):1741--1765, 2004.
\newblock \href{https://doi.org/10.1214/105051604000000639}{doi:10.1214/105051604000000639}.

\bibitem[Lai and Choi(2026)]{LaiChoi2026}
Z.~Lai and M.~C.~H. Choi.
\newblock Information-theoretic coordinate subset and partition selection of multivariate Markov chains via submodular optimization.
\newblock \emph{Journal of Combinatorial Optimization}, 51, article~40, 2026.
\newblock \href{https://doi.org/10.1007/s10878-026-01417-4}{doi:10.1007/s10878-026-01417-4}.

\bibitem[Lim and Choi(2026a)]{LimChoiAdditive2026}
R.~J.~Y. Lim and M.~C.~H. Choi.
\newblock On additive averaging kernels for finite Markov chains.
\newblock \emph{Japanese Journal of Statistics and Data Science}, 2026.
\newblock \href{https://doi.org/10.1007/s42081-026-00373-x}{doi:10.1007/s42081-026-00373-x}.

\bibitem[Lim and Choi(2026b)]{LimChoiTwoBlock2026}
R.~J.~Y. Lim and M.~C.~H. Choi.
\newblock Optimising two-block averaging kernels to speed up Markov chains.
\newblock arXiv:2603.10318, 2026.
\newblock \href{https://doi.org/10.48550/arXiv.2603.10318}{doi:10.48550/arXiv.2603.10318}.

\bibitem[Nock et~al.(2013)Nock, Magdalou, Briys, and Nielsen]{NockEtAl2013}
R.~Nock, B.~Magdalou, E.~Briys, and F.~Nielsen.
\newblock Mining matrix data with Bregman matrix divergences for portfolio selection.
\newblock In F.~Nielsen and R.~Bhatia, editors, \emph{Matrix Information Geometry}, pages 373--402.
\newblock Springer, Berlin, Heidelberg, 2013.
\newblock \href{https://doi.org/10.1007/978-3-642-30232-9_15}{doi:10.1007/978-3-642-30232-9\_15}.

\bibitem[Tsallis(1988)]{Tsallis1988}
C.~Tsallis.
\newblock Possible generalization of Boltzmann-Gibbs statistics.
\newblock \emph{Journal of Statistical Physics}, 52(1--2):479--487, 1988.
\newblock \href{https://doi.org/10.1007/BF01016429}{doi:10.1007/BF01016429}.

\bibitem[Tsallis(2002)]{Tsallis2000}
C.~Tsallis.
\newblock Entropic nonextensivity: a possible measure of complexity.
\newblock \emph{Chaos, Solitons \& Fractals}, 13(3):371--391, 2002.
\newblock \href{https://doi.org/10.1016/S0960-0779(01)00019-4}
{doi:10.1016/S0960-0779(01)00019-4}.

\bibitem[Watanabe and Wolfer(2026)]{WatanabeWolfer2026}
S.~Watanabe and G.~Wolfer.
\newblock Characterization of exponential families of lumpable stochastic matrices.
\newblock \emph{Information Geometry}, 2026.
\newblock \href{https://doi.org/10.1007/s41884-026-00194-7}{doi:10.1007/s41884-026-00194-7}.

\bibitem[Wolfer and Watanabe(2021)]{WolferWatanabe2021}
G.~Wolfer and S.~Watanabe.
\newblock Information geometry of reversible Markov chains.
\newblock \emph{Information Geometry}, 4:393--433, 2021.
\newblock \href{https://doi.org/10.1007/s41884-021-00061-7}{doi:10.1007/s41884-021-00061-7}.

\bibitem[Wolfer and Watanabe(2024)]{WolferWatanabe2024}
G.~Wolfer and S.~Watanabe.
\newblock Geometric aspects of data-processing of Markov chains.
\newblock \emph{Transactions of Mathematics and Its Applications}, 8(1):tnae001, 2024.
\newblock \href{https://doi.org/10.1093/imatrm/tnae001}{doi:10.1093/imatrm/tnae001}.

\end{thebibliography}
\end{document}